\documentclass[12pt, titlepage]{article}

\usepackage[left=1.5in, right=1in, top=1.25in, bottom=1in]{geometry}

\usepackage{amsmath}
\usepackage{amsfonts}
\usepackage{amssymb}
\usepackage{amsthm}
\usepackage{graphicx}
\usepackage{subfig}
\usepackage{tikz}

\newtheorem{thm}{Theorem}[section]
\newtheorem{defn}[thm]{Definition}
\newtheorem{prop}[thm]{Proposition}
\newtheorem{cor}[thm]{Corollary}
\newtheorem*{lem}{Lemma}
\newtheorem{claim}{Claim}

\newcommand{\mc}[1]{\mathcal{#1}}
\newcommand{\ra}{\rightarrow}
\newcommand{\Ra}{\Rightarrow}
\newcommand{\minus}{\,\backslash\,}
\newcommand{\ps}{\;\;}
\DeclareMathOperator{\Ext}{Ext}

\tikzstyle{vertex} = [fill=black, circle, inner sep = 0pt, minimum size = 2mm]
\tikzstyle{edge} = [->, thick, >=latex]
\tikzstyle{arc} = [->, thick, >=latex, bend angle=45, bend right]

\begin{document}
\title{Various Theorems on Tournaments}
\author{Gaku Liu \\
\small Department of Mathematics \\
\small Princeton University \\
\small Senior Thesis \\
\small Adviser: Paul Seymour}
\date{May 7, 2012}

\maketitle

\pagestyle{empty}

\vspace*{\fill}
\section*{Acknowledgments}
None of this work would have been possible without my adviser Paul Seymour, and this thesis very much reflects his ideas, guidance, and contributions. I would also like to thank Maria Chudnovsky for putting in the time and work to be my second reader.
\vspace{2in}
\vspace*{\fill}

\newpage

\begin{abstract}
In this thesis we prove a variety of theorems on tournaments. A \emph{prime} tournament is a tournament $G$ such that there is no $X \subseteq V(G)$, $1 < |X| < |V(G)|$, such that for every vertex $v \in V(G) \minus X$, either $v \ra x$ for all $x \in X$ or $x \ra v$ for all $x \in X$. First, we prove that given a prime tournament $G$ which is not in one of three special families of tournaments, for any prime subtournament $H$ of $G$ with $5 \le |V(H)| < |V(G)|$ there exists a prime subtournament of $G$ with $|V(H)| + 1$ vertices that has a subtournament isomorphic to $H$. We next prove that for any two cyclic triangles $C$, $C^\prime$ in a prime tournament $G$, there is a sequence of cyclic triangles $C_1,\ldots,C_n$ such that $C_1 = C$, $C_n = C^\prime$, and $C_i$ shares an edge with $C_{i+1}$ for all $1 \le i \le n-1$. Next, we consider what we call \emph{matching tournaments}, tournaments whose vertices can be ordered in a horizontal line so that every vertex is the head or tail of at most one edge that points right-to-left. We determine the conditions under which a tournament can have two different orderings satisfying the above conditions. We also prove that there are infinitely many minimal tournaments that are not matching tournaments. Finally, we consider the tournaments $K_n$ and $K_n^\ast$, which are obtained from the transitive tournament with $n$ vertices by reversing the edge from the second vertex to the last vertex and from the first vertex to the second-to-last vertex, respectively. We prove a structure theorem describing tournaments which exclude $K_n$ and $K_n^\ast$ as subtournaments.
\end{abstract}

\newpage

\vspace*{\fill}
\tableofcontents
\vspace{2in}
\vspace*{\fill}

\newpage

\pagestyle{headings}
\setcounter{page}{1}

\section{Introduction}

A \emph{tournament} $G$ is a non-null, loopless directed graph such that for any two distinct vertices $u, v \in V(G)$, there is exactly one edge with both ends in $\{u,v\}$. 
A \emph{subtournament} of a tournament $G$ is a tournament induced on a nonempty subset of $V(G)$. 
If $H$ is a subtournament of $G$ and $X \subseteq V(G)$, we use $H + X$ to denote the subtournament of $G$ induced on $V(H) \cup X$. 
If $X \subsetneq V(G)$, we use $G - X$ to denote the subtournament induced on $V(G) \minus X$. 
We use $H+v$ to mean $H+\{v\}$ and $G-v$ to mean $G - \{v\}$. 
For each vertex $v \in V(G)$, let $A_G(v) = \{u \in V(G) : v \ra u\}$ be the set of \emph{outneighbors} of $v$ in $G$, and let $B_G(v) = \{u \in V(G) : u \ra v\}$ be the set of \emph{inneighbors} of $v$ in $G$. We call $|A_G(v)|$ the \emph{outdegree} of $v$ in $G$ and $|B_G(v)|$ the \emph{indegree} of $v$ in $G$.
For any two disjoint sets $X, Y \subseteq V(G)$, we write $X \Ra Y$ if $x \ra y$ for every $x \in X$ and $y \in Y$. We use $v \Ra X$ to mean $\{v\} \Ra X$.
Given distinct vertices $u,v \in V(G)$, define $d_{uv} \in \{+,-\}$ so that $d_{uv} = +$ if $u \ra v$ and $d_{uv} = -$ if $v \ra u$.
Finally, an \emph{ordering} of $V(G)$ is a list $v_1,\ldots,v_n$ of the vertices of $G$.
A \emph{transitive} tournament is a tournament whose vertices can be ordered $v_1,\ldots,v_n$ such that $v_i \ra v_j$ if $i < j$. We call the unique ordering $v_1,\ldots,v_n$ which satisfies the previous condition the \emph{standard ordering} of the vertices of a transitive tournament. We use $I_n$ to denote the isomorphism class of transitive tournaments with $n$ vertices. We will often refer to $I_n$ as a tournament itself; likewise, when we define other isomorphism classes of tournaments, we will often refer to them as tournaments themselves. 

This thesis contains various results on tournaments proven over the course of the year. While the different major results are largely independent of each other, there are a few common ideas. Section 2 introduces homogeneous sets, the ``substitution'' construction, and prime tournaments, concepts that are central to this paper and will be used throughout. Sections 3 through 6 each showcase a different theorem. Section 3 gives a theorem that allows us to ``grow'' a prime tournament one vertex at a time starting from any of its prime subtournaments. This result strengthens a theorem by Schmerl and Trotter \cite{ST} which is often used in the study of prime tournaments (see for example \cite{BB}, \cite{BDI}, \cite{L}). In Section 4 we prove an interesting result on the structure of cyclic triangles within prime tournaments. Sections 5 and 6 deal with tournaments which are in a sense ``almost-transitive''---they are formed from transitive tournaments by reversing a set of edges satisfying a given property. Section 5 deals with matching tournaments, tournaments for which this set of edges is a matching. Section 6 deals with tournaments with only a single edge reversed, and considers the structure of tournaments that exclude these. The final section considers directions for future research.

While this thesis is meant to be read chronologically, Sections 3 through 6 are more or less independent of each other with the following exceptions: the proofs of Theorems \ref{primeconnected} and \ref{noKn} use the concept of weaves defined in the proof of Theorem \ref{grow}, and Section 6 uses the definition of a backedge given at the beginning of Section 5.

\section{Homogeneous sets and prime tournaments}

\subsection{Basic definitions and properties}

Given a tournament $G$, a \emph{homogeneous set} of $G$ is a subset of vertices $X \subseteq V(G)$ such that for all vertices $v \in V(G) \minus X$, either $v \Ra X$ or $X \Ra v$. A homogeneous set $X \subseteq V(G)$ is \emph{nontrivial} if $1 < |X| < |V(G)|$; otherwise it is \emph{trivial}.

We list some basic properties of homogeneous sets.

\begin{prop}[Restriction] \label{restrict}
If $X$ is a homogeneous set of a tournament $G$, and $H$ is a subtournament of $G$, then $X \cap V(H)$ is a homogeneous set of $H$.
\end{prop}

\begin{prop}[Extension] \label{extend}
If $H_1$, $H_2$ are subtournaments of a tournament $G$ and $X \subseteq V(H_1) \cap V(H_2)$ is a homogeneous set of both $H_1$ and $H_2$, then $X$ is a homogeneous set of the tournament induced on $V(H_1) \cup V(H_2)$.
\end{prop}

\begin{prop}[Cloning] \label{clone}
Let $G$ be a tournament, $x,y \in V(G)$ be distinct vertices, and $X \subseteq V(G) \minus \{y\}$. If $X$ is a homogeneous set of $G - y$ and $\{x,y\}$ is a homogeneous set of $G$, then $X$ is a homogeneous set of $G$.
\end{prop}

\begin{prop}[Intersection] \label{intersect}
If $X$, $Y$ are homogeneous sets of a tournament $G$, then $X \cap Y$ is a homogeneous set of $G$.
\end{prop}

\begin{prop}[Subtraction] \label{subtract}
If $X$, $Y$ are homogeneous sets of a tournament $G$ and $Y \minus X \neq \varnothing$, then $X \minus Y$ is a homogeneous set of $G$.
\end{prop}

\begin{prop}[Union] \label{union}
Suppose $G$ is a tournament and $X, Y \subseteq V(G)$ such that $X \cap Y \neq \varnothing$, $X$ is a homogeneous set of $G - (Y \minus X)$, and $Y$ is a homogeneous set of $G - (X \minus Y)$. Then $X \cup Y$ is a homogeneous set of $G$.
\end{prop}

A tournament is \emph{prime} if all of its homogeneous sets are trivial; otherwise, it is \emph{decomposable}.
Given a tournament $G$ and an ordering $v_1,\ldots,v_n$ of its vertices, and given tournaments $H_1,\ldots,H_n$, let $G(H_1,\ldots,H_n)$ be a tournament with vertex set $V_1 \cup V_2 \cup \cdots \cup V_n$, where the $V_i$ are pairwise disjoint sets of vertices, such that $V_i \Ra V_j$ if $v_i \ra v_j$, and for all $1 \le i \le n$ the subtournament of $G(H_1,\ldots,H_n)$ induced on $V_i$ is isomorphic to $H_i$. Every tournament with at least two vertices can be written as $G^\prime(H_1,\ldots,H_n)$ where $G^\prime$ is a prime tournament with at least two vertices. The prime tournaments are precisely those tournaments $G$ which cannot be written as $G^\prime(H_1,\ldots,H_n)$ for some $G^\prime$ with $2 \le |V(G^\prime)| < |V(G)|$.

A tournament $G$ is \emph{strongly connected} if for any two vertices $u,v \in V(G)$, there is a directed path from $u$ to $v$ and a directed path from $v$ to $u$. A \emph{strongly connected component}, or \emph{strong component}, of a tournament $G$ is a maximal strongly connected subtournament of $G$. The strong components of a tournament $G$ can be ordered as $S_1,\ldots, S_s$ so that $G$ can be written as $I_s(S_1,\ldots,S_s)$, where $I_s$ has the standard ordering for a transitive tournament. From this we see that the vertices of a strong component of a tournament form a homogeneous set, so if a tournament is prime, either it is strongly connected or all of its strong components have only one vertex. In the latter case the tournament is transitive, and $I_n$ has a homogeneous set for $n \ge 3$.
So every prime tournament with $\ge 3$ vertices is strongly connected.

The tournament with 1 vertex and the tournament with 2 vertices are both prime. The only prime tournament with 3 vertices is the cyclic triangle. It can be checked that all tournaments with 4 vertices are decomposable. For 5 vertices, there are exactly three prime tournaments $T_5$, $U_5$, and $W_5$, drawn below.

\begin{figure}[htbp]
\centering
\subfloat[$T_5$]{
\begin{tikzpicture}[scale=1.5]
\node[vertex] (v1) at (90:1) {};
\node[vertex] (v2) at (162:1) {};
\node[vertex] (v3) at (234:1) {};
\node[vertex] (v4) at (306:1) {};
\node[vertex] (v5) at (18:1) {};
\draw[edge] (v1) to (v2);
\draw[edge] (v1) to (v3);
\draw[edge] (v2) to (v3);
\draw[edge] (v2) to (v4);
\draw[edge] (v3) to (v4);
\draw[edge] (v3) to (v5);
\draw[edge] (v4) to (v5);
\draw[edge] (v4) to (v1);
\draw[edge] (v5) to (v1);
\draw[edge] (v5) to (v2);
\end{tikzpicture}
} \hspace{1.5cm}
\subfloat[$U_5$]{
\begin{tikzpicture}[scale=1.5]
\node[vertex] (v1) at (90:1) {};
\node[vertex] (v2) at (162:1) {};
\node[vertex] (v3) at (234:1) {};
\node[vertex] (v4) at (306:1) {};
\node[vertex] (v5) at (18:1) {};
\draw[edge] (v2) to (v1); %%
\draw[edge] (v1) to (v3);
\draw[edge] (v2) to (v3);
\draw[edge] (v2) to (v4);
\draw[edge] (v3) to (v4);
\draw[edge] (v3) to (v5);
\draw[edge] (v4) to (v5);
\draw[edge] (v4) to (v1);
\draw[edge] (v5) to (v1);
\draw[edge] (v5) to (v2);
\end{tikzpicture}
} \hspace{1.5cm}
\subfloat[$W_5$]{
\begin{tikzpicture}[scale=1.5]
\node[vertex] (v1) at (90:1) {};
\node[vertex] (v2) at (162:1) {};
\node[vertex] (v3) at (234:1) {};
\node[vertex] (v4) at (306:1) {};
\node[vertex] (v5) at (18:1) {};
\draw[edge] (v2) to (v1); %%
\draw[edge] (v1) to (v3);
\draw[edge] (v2) to (v3);
\draw[edge] (v2) to (v4);
\draw[edge] (v4) to (v3); %%
\draw[edge] (v3) to (v5);
\draw[edge] (v4) to (v5);
\draw[edge] (v4) to (v1);
\draw[edge] (v5) to (v1);
\draw[edge] (v5) to (v2);
\end{tikzpicture}
}
\caption{The three five-vertex prime tournaments}
\end{figure}
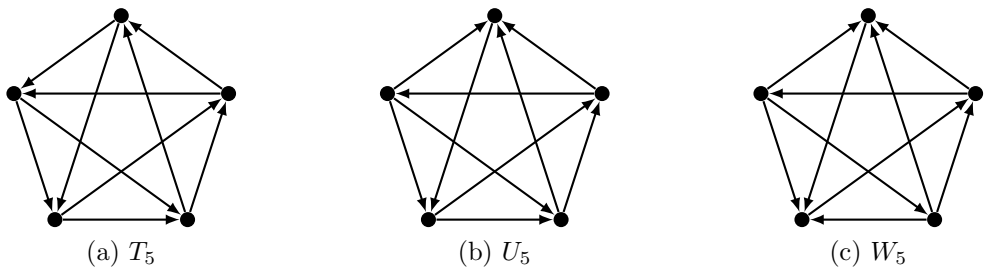

These three tournaments can be generalized to any odd number of vertices as follows.

\begin{defn} \label{TUW}
Let $k \ge 0$ and $n = 2k+1$. Define the tournaments $T_n$, $U_n$, and $W_n$ as follows.
\begin{itemize}
\item $T_n$ is the tournament with vertices $v_1, \ldots, v_n$ such that $v_i \ra v_j$ if $j \equiv i+1, i+2, \ldots,$ or $i+k \pmod{n}$. 
\item $U_n$ is the tournament obtained from $T_n$ by reversing all edges which have both ends in $\{v_1,\ldots,v_k\}$. 
\item $W_n$ is the tournmanet with vertices $w_1, \ldots, w_n$ such that $w_i \ra w_j$ if $1 \le i < j \le n-1$, and $\{w_2, w_4, \ldots, w_{n-1}\} \Ra w_n \Ra \{w_1, w_3, \ldots, w_{n-2}\}$.
\end{itemize}
\end{defn}

Note the following degenerate cases: $T_1$, $U_1$, and $W_1$ are all the single-vertex tournament, and $T_3$, $U_3$, and $W_3$ are all the cyclic triangle.

It can checked that $T_n$, $U_n$, and $W_n$ are prime for all odd $n$. Also, note that $T_n$, $U_n$, and $W_n$ have subtournaments isomorphic to $T_m$, $U_m$, and $W_m$, respectively, for $m \le n$. In fact, the only prime subtournaments of $T_n$, $U_n$, and $W_n$ with at least 3 vertices are $T_m$, $U_m$, and $W_m$, respectively, for $3 \le m \le n$.

\subsection{Prime subtournaments of prime tournaments}

Given a prime tournament, it is natural to ask about its prime subtournaments.
There have been several results proven about prime subtournaments; to state them, we introduce some extra terminology. Let $G$ be a tournament and $H$ be a prime subtournament of $G$ with $|V(H)| \ge 3$. Define $\Ext(H)$ to be the set of vertices $v \in V(G) \minus V(H)$ for which the tournament induced on $H+v$ is prime. Let $Z(H)$ be the set of vertices $v \in V(G) \minus V(H)$ for which $V(H)$ is a homogeneous set of $H+v$. Finally, for each $x \in V(H)$, let $V_x(H)$ be the set of vertices $v \in V(G) \minus V(H)$ for which $\{v,x\}$ is a homogeneous set of $H+v$.

We can use the basic properties of homogeneous tournaments found in Propositions \ref{restrict} through \ref{union} to prove the following, which appears in \cite{ER}.

\begin{prop}[Ehrenfeucht and Rozenberg \cite{ER}] \label{ERlemma}
Let $G$ be a tournament and $H$ be a prime subtournament of $G$ with $|V(H)| \ge 3$. Then
\renewcommand{\labelenumi}{(\roman{enumi})}
\begin{enumerate}
\item The collection of sets $\{\Ext(H), Z(H)\} \cup \{V_x(H) : x \in V(H)\}$ forms a partition of $V(G) \minus V(H)$.
\item If $u,v \in \Ext(H)$, $u \neq v$, such that $H + \{u,v\}$ is decomposable, then $\{u,v\}$ is a homogeneous set of $H + \{u,v\}$.
\item If $u \in Z(H)$ and $v \in (V(G) \minus V(H)) \minus Z(H)$ such that $H + \{u,v\}$ is decomposable, then $V(H) \cup \{v\}$ is a homogeneous set of $H + \{u,v\}$. 
\item If $u \in V_x(H)$ for some $x \in V(H)$ and $v \in (V(G) \minus V(H)) \minus V_x(H)$ such that $H + \{u,v\}$ is decomposable, then $\{u,x\}$ is a homogeneous set of $H + \{u,v\}$.
\end{enumerate}
\end{prop}

\begin{proof}
We will only prove part (i); the proofs of the other parts involve similar techniques. We need to show that if $v \in V(G) \minus V(H)$ is a vertex such that $H + v$ is decomposable, then $v$ is in exactly one of the sets $Z(H)$, $V_x(H)$ for $x \in V(H)$. First, we show that $v$ is in at least one of these sets. Since $H+v$ is decomposable, it has a nontrivial homogeneous set $X \subseteq V(H+v)$. By restriction to $H$, we have that $X \cap V(H)$ is a homogeneous set of $H$. Since $H$ is prime, we must thus have either $|X \cap V(H)| \le 1$ or $X \cap V(H) = V(H)$. If the former case, then since $X$ is a nontrivial homogeneous set of $H+v$, we must have $|X \cap V(H)| = 1$ and $v \in X$, so $v \in V_{X \cap V(H)}(H)$, as desired. If the latter case, then $v \in Z(H)$, as desired. So $v$ is in at least one of $Z(H)$, $V_x(H)$ for $x \in V(H)$.

Suppose $v$ is in at least two of these sets. First, suppose $v \in Z(H)$ and $v \in V_x(H)$ where $x \in V(H)$. We have that $V(H)$ and $\{v,x\}$ are homogeneous sets of $H+v$. Applying Proposition \ref{subtract}, we have that $V(H) \minus \{x\}$ is a homogeneous set of $H+v$, and restricting to $H$ we have that $V(H) \minus \{x\}$ is a homogeneous set of $H$. Since $|V(H)| \ge 3$, this is a nontrivial homogeneous set of $H$, a contradiction.

Now suppose $v \in V_x(H)$ and $v \in V_y(H)$ for two distinct $x,y \in V(H)$. Then $\{v,x\}$ and $\{v,y\}$ are homogeneous sets of $H+v$. Applying Proposition \ref{union}, we have that $\{v,x,y\}$ is a homogeneous set of $H+v$, and restricting to $H$, we have that $\{x,y\}$ is a homogeneous set of $H$. Since $|V(H)| \ge 3$, this is a nontrivial homogeneous set of $H$, a contradiction. This completes the proof.
\end{proof}

This result leads to the following corollaries in the case where $G$ is prime.

\begin{cor}[Ehrenfeucht and Rozenberg \cite{ER}] \label{grow2}
Let $G$ be a prime tournament and $H$ be a prime subtournament of $G$ with $3 \le |V(H)| \le |V(G)| - 2$. Then there exist distinct vertices $u,v \in V(G) \minus V(H)$ such that $H + \{u,v\}$ is prime.
\end{cor}

\begin{proof}
Suppose that for every distinct $u,v \in V(G) \minus V(H)$, $H+\{u,v\}$ is decomposable. We first prove that $Z(H)$ and $V_x(H)$ are empty for all $x \in V(H)$. By Proposition \ref{ERlemma}(iii), we have that $V(H) \cup \{v\}$ is a homogeneous set of $H + \{u,v\}$ for every $u \in Z(H)$ and $v \in (V(G) \minus V(H)) \minus Z(H)$. Repeatedly applying Propositions \ref{extend} and \ref{union} on these homogeneous sets, we thus have that $V(G) \minus Z(H)$ is a homogeneous set of $G$. If $Z(H)$ is nonempty, then $V(G) \minus Z(H)$ is a nontrivial homogeneous set of $G$, contradicting the fact that $G$ is prime. So $Z(H)$ is empty.

Similarly, for each $x \in V(H)$, we have by Proposition \ref{ERlemma}(iv) that $\{u,x\}$ is a homogeneous set of $H + \{u,v\}$ for every $u \in V_x(H)$, $v \in (V(G) \minus V(H)) \minus V_x(H)$. Repeatedly applying Propositions \ref{extend} and \ref{union} on these homogeneous sets, we have that $V_x(H) \cup \{x\}$ is a homogeneous set of $G$. If $V_x(H)$ is nonempty, then $V_x(H) \cup \{x\}$ is a nontrivial homogeneous set of $G$, a contradiction. So $V_x(H)$ is empty for each $x \in V(H)$.

It follows by Proposition \ref{ERlemma}(i) that $V(G) \minus V(H) = \Ext(H)$. Thus, by Proposition \ref{ERlemma}(ii), $\{u,v\}$ is a homogeneous set of $H + \{u,v\}$ for every distinct $u,v \in V(G) \minus V(H)$. Since $|V(G) \minus V(H)| \ge 2$, we can apply Proposition \ref{union} repeatedly on these homogeneous sets to get that $V(G) \minus V(H)$ is a homogeneous set of $G$. This is a nontrivial homogeneous set of $G$, a contradiction.
\end{proof}

\begin{cor} \label{contains5}
Every prime tournament $G$ with $|V(G)| \ge 5$ has a prime subtournament with 5 vertices.
\end{cor}

\begin{proof}
Since $G$ is prime and has $> 2$ vertices, it is not transitive, so it contains a cyclic triangle. Applying Corollary \ref{grow2} with $H$ as a cyclic triangle gives the corollary.
\end{proof}

Corollary \ref{grow2} can be thought of as a ``growing'' lemma: within a prime tournament, we can grow an increasing sequence of prime subtournaments so that each subtournament contains the previous one. For example, starting with a cyclic triangle in a prime tournament $G$ and repeatedly applying the corollary, we have that $G$ contains a prime subtournament with $n$ vertices for every odd $n < |V(G)|$. In particular, $G$ contains a prime subtournament with either $|V(G)| - 2$ or $|V(G)| - 1$ vertices. This last statement was improved upon by Schmerl and Trotter in \cite{ST}.

\begin{thm}[Schmerl and Trotter \cite{ST}] \label{critical}
If $G$ is a prime tournament with $|V(G)| \ge 6$, and $G$ is not $T_n$, $U_n$, or $W_n$ for any odd $n$, then $G$ has a prime subtournament with $|V(G)| - 1$ vertices.
\end{thm}

\begin{thm}[Schmerl and Trotter \cite{ST}] \label{shrink2}
If $G$ is a prime tournament with $|V(G)| \ge 7$, then $G$ has a prime subtournament with $|V(G)| - 2$ vertices.
\end{thm}

Schmerl and Trotter's proof of these two theorems involve ``growing'' prime subtournaments, during which Corollary \ref{grow2} is essential. Our result in the next section can be thought of as a strengthening of Corollary \ref{grow2} in that it allows us to grow the sequence of prime subtournaments one vertex at a time instead of two vertices at a time in the case where $G$ is not $T_n$, $U_n$, or $W_n$. In particular, our theorem has both Theorems \ref{critical} and \ref{shrink2} as immediate corollaries.

\section{Growing prime tournaments}

\subsection{Statement of theorem}

The main theorem is as follows.

\begin{thm} \label{grow}
Let $G$ be a prime tournament which is not $T_n$, $U_n$, or $W_n$ for any odd $n$, and let $H$ be a prime subtournament of $G$ with $5 \le |V(H)| \le |V(G)| - 1$. Then there exists a prime subtournament of $G$ with $|V(H)|+1$ vertices that has a subtournament isomorphic to $H$.
\end{thm}

Note that Theorem \ref{grow} is not a strict strengthening of Corollary \ref{grow2} because the theorem does not guarantee that the subtournament with $|V(H)|+1$ vertices contains the actual vertices of $H$; it only guarantees that it contains a subtournament isomorphic to $H$. However, in many applications where one would want to grow prime subtournaments, only the isomorphism class of the previous subtournament matters; for example, to use this theorem to prove Theorems \ref{critical} and \ref{shrink2}, only the number of vertices at each step matters.

Theorem \ref{grow} is based on and is a direct analogue of a theorem by Chudnovsky and Seymour for undirected graphs, found in \cite{CS}. Indeed, the proof for Thoerem \ref{grow} found here is closely related to the proof found in \cite{CS}. Chudnovsky and Seymour used their theorem to develop a polynomial-time algorithm to find simplicial cliques in prime claw-free graphs.

\subsection{Proof of theorem}

Or proof consists of two main claims.

\begin{claim}
The theorem holds when $|V(H)| \ge |V(G)| - 2$.
\end{claim}

\begin{claim}
If $n \ge 7$ is odd and $G$ is a prime tournament with $n+1$ vertices that has a subtournament isomorphic to $T_{n}$ ($U_{n}$, $W_{n}$, respectively), then $G$ has a prime subtournament with $n-1$ vertices that has a subtournament isomorphic to $T_{n-2}$ ($U_{n-2}$, $W_{n-2}$, respectively).
\end{claim}

Assuming the truth of these two claims, we can prove Theorem \ref{grow} as follows: By Claim 1, we can assume $|V(H)| < |V(G)| - 2$. First suppose that $H$ is not $T_n$, $U_n$, or $W_n$ for any $n$. By Corollary \ref{grow2}, there are vertices $\{u,v\} \in V(G) \minus V(H)$ such that $H + \{u,v\}$ is prime. Since $H$ is not $T_n$, $U_n$, or $W_n$ for any $n$, $H + \{u,v\}$ is not any of these tournaments either (see the second paragraph after Definition \ref{TUW}). So applying Claim 1 to $H+\{u,v\}$ with subtournament $H$, we have a subtournament of $H+\{u,v\}$ with $|V(H)|+1$ vertices that has a subtournament isomorphic to $H$, as desired.

Now assume $H$ is $T_n$, $U_n$, or $W_n$ for some odd $n \ge 5$. We will assume $H$ is $T_n$; the arguments for $U_n$ and $W_n$ are identical. Suppose there is no prime subtournament of $G$ with $n+1$ vertices that has a subtournament isomorphic to $T_n$. Let $m$ be the largest odd integer such that
\begin{itemize} 
\item $G$ has a subtournament isomorphic to $T_m$, and 
\item $G$ has no prime subtournament with $m+1$ vertices that has a subtournament isomorphic to $T_m$. 
\end{itemize}
Thus, $m \ge n \ge 5$. By Claim 1, we have $m < |V(G)| - 2$. We claim that $G$ has a subtournament isomorphic to $T_{m+2}$. Applying Corollary \ref{grow2} on a copy of $T_m$ in $G$, we have that $G$ has a prime subtournament $H_{m+2}$ with $m+2$ vertices that has a subtournament isomorphic to $T_m$. If $H_{m+2}$ is not $T_{m+2}$, then we can apply Claim 1 to it to obtain a prime subtournament with $m+1$ vertices that has a subtournament isomorphic to $T_m$, contradicting the definition of $m$. Thus, $H_{m+2}$ is $T_{m+2}$.

Thus, $G$ has a subtournament isomorphic to $T_{m+2}$. Hence, by the maximality of $m$, there is a prime subtournament $H_{m+3}$ of $G$ with $m+3$ vertices that has a subtournament isomorphic to $T_{m+2}$. Applying Claim 2, $H_{m+3}$ has a prime subtournament with $m+1$ vertices that has a subtournament isomorphic to $T_m$. This contradicts the definition of $m$, completing the proof.

We now prove the two claims.

\begin{proof}[Proof of Claim 1]
The proof of this claim is made considerably simpler by establishing the right definitions, so we will spend a good amount of time doing so. These definitions will also be useful in later proofs.

We define a \emph{weave} $\langle w_1, \ldots, w_n \rangle$ to be a tournament with vertices $w_1, \ldots, w_n$ such that
\renewcommand{\labelenumi}{(\arabic{enumi})}
\renewcommand{\labelenumii}{(\arabic{enumi}\alph{enumii})}
\begin{enumerate}
\item $w_i \ra w_j$ if $i < j$ and $i,j$ have opposite parity
\item One of the following holds:
	\begin{enumerate}
	\item $w_i \ra w_j$ for all $i < j$ with $i,j$ odd
	\item $w_j \ra w_i$ for all $i < j$ with $i,j$ odd
	\end{enumerate}
\item One of the following holds:
	\begin{enumerate}
	\item $w_i \ra w_j$ for all $i < j$ with $i,j$ even
	\item $w_j \ra w_i$ for all $i < j$ with $i,j$ even
	\end{enumerate}
\end{enumerate}
Using ``F'' to mean ``forward'' and ``B'' to mean ``backward,'' we will call a weave an FF weave if (2a) and (3a) hold, an FB weave if (2a) and (3b) hold, a BF weave if (2b) and (3a) hold, and a BB weave if (2b) and (3b) hold. We refer to these as the four \emph{types} of weaves. (A weave can be of more than one type if $n \le 3$.)

The following facts can be easily checked; the information that is most relevant to our proof is summarized in the corollary afterwards.

\begin{prop} \label{weave}
Let $W = \langle w_1, \ldots, w_n \rangle$ be a weave.
\begin{itemize}
\item If $W$ is FF, then it is transitive.
\item If $n$ is even and $W$ is FB, then $\{w_2,\ldots,w_n\}$ is a homogeneous set of $W$.
\item If $n$ is even and $W$ is BF, then $\{w_1,\ldots,w_{n-1}\}$ is a homogeneous set of $W$.
\item If $n$ is even and $W$ is BB, then $\{w_1,w_n\}$ is a homogeneous set of $W$.
\item If $n$ is odd and $W$ is FB, then $\{w_1,\ldots,w_{n-1}\}$ is a homogeneous set of $W$.
\item If $n$ is odd and $W$ is BF, then $W$ is $U_n$.
\item If $n$ is odd and $W$ is BB, then $W$ is $T_n$. 
\end{itemize}
Let $v \notin V(W)$ be a vertex such that $v \ra w_i$ for odd $i$ and $w_i \ra v$ for even $i$.
\begin{itemize}
\item If $n$ is odd and $W$ is FF or FB, then $\{v,w_1,\ldots,w_{n-1}\}$ is a homogeneous set of $W+v$.
\item If $n$ is odd and $W$ is BF or BB, then $\{v,w_n\}$ is a homogeneous set of $W+v$.
\item If $n$ is even and $W$ is FF, then $W+v$ is $W_{n+1}$.
\item If $n$ is even and $W$ is FB or BF, then $W+v$ is $U_{n+1}$.
\item If $n$ is even and $W$ is BB, then $W+v$ is $T_{n+1}$.
\end{itemize}
\end{prop}

\begin{cor} \label{primeweave}
Let $W$ be a weave.
\renewcommand{\labelenumi}{(\roman{enumi})}
\begin{enumerate}
\item If $|V(W)| \ge 3$ and $W$ is prime, then $W$ is $T_n$ or $U_n$ for some $n$.
\item If $|V(W)| \ge 2$ and $v \notin V(W)$ is a vertex as in Proposition \ref{weave}, and $W+v$ is prime, then $W+v$ is $T_n$, $U_n$, or $W_n$ for some $n$.
\end{enumerate}
\end{cor}

The following fact will also be important.

\begin{prop} \label{delpair}
If $W = \langle w_1, \ldots, w_n \rangle$ is a weave and $1 \le i \le n-1$, then $W - \{w_i,w_{i+1}\}$ is a weave of the same type as $W$.
\end{prop}

\begin{cor} \label{weaveiso}
If $W = \langle w_1, \ldots, w_n \rangle$ is a weave, then the subtournaments $W - \{w_1,w_2\},\, W - \{w_2,w_3\},\, \ldots\, ,\, W - \{w_{n-1},w_n\}$ are all isomorphic to each other. Furthermore, there is an isomorphism $\phi$ between $W - \{w_i,w_{i+1}\}$ and $W - \{w_j,w_{j+1}\}$ such that if $\phi(w_k) = w_\ell$, then $k$ and $\ell$ have the same parity.
\end{cor}

We are now ready to prove the claim. The case $|V(H)| = |V(G)| - 1$ is trivial, so assume $|V(H)| = |V(G)| - 2$. Suppose there is no prime subtournament of $G$ with $|V(H)| + 1$ vertices that has a subtournament isomorphic to $H$.
Let $\{u,v\} = V(G) \minus V(H)$, where $u \ra v$. We will call a weave $W = \langle w_1, \ldots, w_n \rangle$ a \emph{$u,v$-weave} if all of the following hold:
\renewcommand{\labelenumi}{(\alph{enumi})}
\begin{enumerate}
\item $W$ is a subtournament of $G$.
\item $u = w_j$ and $v = w_{j+1}$ for some $1 \le j \le n-1$.
\item $\{w_1,w_3,\ldots\}$ is a homogeneous set in $G - \{w_2,w_4,\ldots\}$.
\item $\{w_2,w_4,\ldots\}$ is a homogeneous set in $G - \{w_1,w_3,\ldots\}$.
\end{enumerate}
Since $\langle u,v \rangle$ is a $u,v$-weave, at least one $u,v$-weave with $\ge 2$ vertices exists. Let $W = \langle w_1, \ldots, w_n \rangle$ be a $u,v$-weave which maximizes $V(W)$. If $W = G$, then by Corollary \ref{primeweave} $G$ is $T_n$ or $U_n$, contradicting the assumptions of the theorem. So $|V(W)| < |V(G)|$.

Now, by Corollary \ref{weaveiso}, $W - \{w_i, w_{i+1}\}$ is isomorphic to $W - \{u,v\}$ for all $1 \le i \le n-1$. In fact, because of the second sentence of Corollary \ref{weaveiso} and (c) and (d) in the definition of a $u,v$-weave, we have that $G - \{w_i,w_{i+1}\}$ is isomorphic to $G - \{u,v\} = H$ for all $1 \le i \le n-1$.

Let $H^\prime = G - \{w_1,w_2\}$. Since $H^\prime$ is isomorphic to $H$, by our original assumption $H^\prime + w_2$ must be decomposable. It follows that either $w_2 \in Z(H^\prime)$ or $w_2 \in V_x(H^\prime)$ for some $x \in V(H^\prime)$. 
Suppose $w_2 \in V_x(H^\prime)$ for some $x \in V(H^\prime)$. We thus have
\begin{itemize}
\item $x \Ra \{w_3, w_5, \ldots\} \minus \{x\}$, and 
\item if $n \ge 4$, then $x \Ra \{w_4, w_6, \ldots\} \minus \{x\}$ if $W$ is FF or BF and $\{w_4, w_6, \ldots\} \minus \{x\} \Ra x$ if $W$ is FB or BB.
\end{itemize}

First suppose that $x \in V(G) \minus V(W)$. Since $\{x,w_2\}$ is a homogeneous set of $H^\prime + w_2 = G - w_1$ but it is not a homogeneous set of $G$, and $w_1 \ra w_2$, we must have $x \ra w_1$. Thus, $x \Ra \{w_1, w_3, \ldots\}$. Furthermore, by (d) in the definition of a $u,v$-weave, we have that $x \Ra \{w_2, w_4, \ldots\}$ if $x \Ra \{w_4, w_6, \ldots\}$ and $\{w_2, w_4, \ldots\} \Ra x$ if $\{w_4, w_6, \ldots\} \Ra x$. Finally, by Proposition \ref{union}, $\{x,w_2,w_4,\ldots\}$ is a homogeneous set in $G - \{w_1,w_3,\ldots\}$. Thus, $\langle x, w_1, \ldots, w_n \rangle$ is a $u,v$-weave. This contradicts the maximality of $W$, so $x \notin V(G) \minus V(W)$.

Now suppose $x \in V(W)$. We cannot have $x = w_i$ for any even $i \ge 4$ because $x \Ra \{w_3, w_5, \ldots\} \minus \{x\}$ but $w_i \ra w_3$ for all even $i \ge 4$. So $x = w_i$ for some odd $i \ge 3$. Now, since $w_2 \in V_{w_i}(H^\prime)$, for every $v \in V(G) \minus V(W)$ we have $d_{vw_2} = d_{vw_i}$. Thus, since $i$ is odd, from (c) and (d) in the definition of a $u,v$-weave we have that $V(W)$ is a homogeneous set of $G$. Since $2 \le |V(W)| < |V(G)|$ as mentioned earlier, this contradicts the primeness of $G$.

Thus, we cannot have $w_2 \in V_x(H^\prime)$ for any $x \in V(H^\prime)$. So $w_2 \in Z(H^\prime)$. By symmetry, we also have that $w_{n-1} \in Z(H^{\prime\prime})$, where $H^{\prime\prime} = G - \{w_{n-1},w_n\}$. If $n = 2$, then we have $w_1,w_2 \in Z(G - \{w_1,w_2\})$, and hence $G - \{w_1,w_2\}$ is a homogeneous set of $G$, a contradiction since $|V(G)| \ge 5$. So assume $n \ge 3$. Then since $w_2 \ra w_3$ and $w_2 \in Z(H^\prime)$, we have $w_2 \Ra V(H^\prime)$. Similarly since $w_{n-2} \ra w_{n-1}$ and $w_{n-1} \in Z(H^{\prime\prime})$, we have $V(H^{\prime\prime}) \Ra w_{n-1}$. In particular, we have $w_2 \Ra V(G) \minus V(W) \Ra w_{n-1}$. If $n$ is odd, then (c) and (d) from the definition of a $u,v$-weave imply that $\{w_2,w_4,\ldots\} \Ra V(G) \minus V(W) \Ra \{w_2,w_4,\ldots\}$, a contradiction since $V(G) \minus V(W)$ is nonempty. So $n$ is even, and (c) and (d) imply that 
\[
\{w_2,w_4,\ldots\} \Ra V(G) \minus V(W) \Ra \{w_1,w_3,\ldots\}.
\]

Thus, $V(G) \minus V(W)$ is a homogeneous set of $G$, and as we showed before it is nonempty. Hence, $|V(G) \minus V(W)| = 1$. Let $\{v\} = V(G) \minus V(W)$. Then $G = W+v$, and by Corollary \ref{primeweave}, $G$ must be $T_{n+1}$, $U_{n+1}$, or $W_{n+1}$. This is a contradiction, completing the proof of Claim 1.
\end{proof}

\begin{proof}[Proof of Claim 2]
We first prove the following general proposition.

\begin{prop} \label{lemma2}
Let $G$ be a prime tournament and let $u \in V(G)$. Let $H_1$, $H_2$ be prime subtournaments of $G$ such that $V(H_1) \cup V(H_2) = V(G) \minus \{u\}$, $H_1 + u$ and $H_2 + u$ are decomposable, and the subtournament $H_{1,2}$ induced on $V(H_1) \cap V(H_2)$ is prime with $|V(H_{1,2})| \ge 3$. Then either $u \in V_x(H_1)$ for some $x \in V(H_1) \minus V(H_2)$ or $u \in V_y(H_2)$ for some $y \in V(H_2) \minus V(H_1)$.
\end{prop}

\begin{proof}
Since $H_1$ is prime and $H_1+u$ is decomposable, we have either $u \in Z(H_1)$ or $u \in V_x(H_1)$ for some $x \in H_1$, and likewise for $H_2$.

First, suppose $u \in Z(H_1)$ and $u \in Z(H_2)$. Then $V(H_1)$ is a homogeneous set of $H_1+u$ and $V(H_2)$ is a homogeneous set of $H_2+u$. Since $V(H_1) \cap V(H_2) \neq \varnothing$ by assumption, by Proposition \ref{union} we have that $V(H_1) \cup V(H_2) = V(G) \minus \{u\}$ is a homogeneous set of $G$. This contradicts the primeness of $G$, so we cannot have this case.

Next, suppose $u \in Z(H_1)$ and $u \in V_y(H_2)$ for some $y \in H_2$. If $y \in V(H_2) \minus V(H_1)$ then we are done, so assume $y \in V(H_1)$, and hence $y \in V(H_{1,2})$. Now, $\{u,y\}$ is a homogeneous set of $H_2+u$ and $V(H_1)$ is a homogeneous set of $H_1+u$; restricting to the subtournament $H_{1,2} + u$, we have that $\{u,y\}$ and $V(H_{1,2})$ are homogeneous sets of $H_{1,2} + u$. Applying Proposition \ref{subtract}, we have that $V(H_{1,2}) \minus \{y\}$ is a homogeneous set of $H_{1,2} + u$, and hence $V(H_{1,2}) \minus \{y\}$ is a homogeneous set of $H_{1,2}$. Since $|V(H_{1,2})| \ge 3$, this contradicts the primeness of $H_{1,2}$.

Finally, suppose $u \in V_x(H_1)$ for some $x \in V(H_1)$ and $u \in V_y(H_2)$ for some $y \in V(H_2)$. If either $x \in V(H_1) \minus V(H_2)$ or $y \in V(H_2) \minus V(H_1)$ then we are done, so assume $x,y \in V(H_{1,2})$. If $x = y$, then $\{u,x\}$ is a homogeneous set of both $H_1+u$ and $H_2+u$, so by Proposition \ref{extend}, $\{u,x\}$ is a homogeneous set of $G$, a contradiction. If $x \neq y$, then restricting to $H_{1,2} + u$ and applying Proposition \ref{union}, we have that $\{u,x,y\}$ is a homogeneous set of $H_{1,2} + u$, so $\{x,y\}$ is a homogeneous set of $H_{1,2}$. As before, this is a contradiction, which completes the proof. 
\end{proof}

We can now prove Claim 2. Assume the hypotheses of Claim 2. Let $H$ be a subtournament of $G$ which is isomorphic to $T_n$, $U_n$, or $W_n$. Let $H^\prime $ and $H^{\prime\prime}$ be prime subtournaments of $H$ with $n-2$ and $n-4$ vertices, respectively. In other words, if $H$ is isomorphic to $T_n$, then $H^\prime$ is isomorphic to $T_{n-2}$ and $H^{\prime\prime}$ is isomorphic to $T_{n-4}$; likewise for $U_n$ and $W_n$.
We wish to prove there is a prime subtournment of $G$ with $n-1$ vertices that has a subtournament isomorphic to $H^\prime$.

Suppose the contrary. Let $\{u\} = V(G) \minus V(H)$. By Proposition \ref{weave}, we can write $H$ as $W+v$, where $W = \langle w_1, \ldots, w_{n-1} \rangle$ is a weave and $v$ is as in Proposition \ref{weave}. For distinct integers $1 \le i_1,\ldots,i_r \le n-1$,
let $H_{i_1,\ldots,i_r} = H - \{w_{i_1},\ldots,w_{i_r}\}$. Note that $H_{i,i+1}$ is isomorphic to $H^\prime$ for all $1 \le i \le n-2$. Hence, by assumption, $H_{i,i+1} + u$ must be decomposable for all $1 \le i \le n-2$.

Now, for distinct integers $1 \le i,j \le n-1$, we define $x_{i,j} \in V(H_{i,j}) \cup \{\infty\}$ as follows: If $u \in V_x(H_{i,j})$ for some $x \in V(H_{i,j})$, then let $x_{i,j} = x$; otherwise, let $x_{i,j} = \infty$. By Proposition \ref{ERlemma}(i), $x_{i,j}$ is well-defined.
Now, suppose $1 \le i,j \le n-2$ are integers such that $\{i,i+1\} \cap \{j,j+1\} = \varnothing$. Then $H_{i,i+1,j,j+1}$ is isomorphic to $H^{\prime\prime}$, and is hence prime. Thus, applying Proposition \ref{lemma2} with $H_{i,i+1}$ as $H_1$ and $H_{j,j+1}$ as $H_2$, we have that for all such $i,j$, either $x_{i,i+1} \in \{w_j,w_{j+1}\}$ or $x_{j,j+1} \in \{w_i,w_{i+1}\}$. We will denote this fact as ($\ast$).

Applying ($\ast$) with $i = 1$ and $j = n-2$, we have that either $x_{1,2} \in \{w_{n-2},w_{n-1}\}$ or $x_{n-2,n-1} \in \{w_1,w_2\}$. Without loss of generality, assume 
\[
x_{1,2} \in \{w_{n-2},w_{n-1}\}.
\]
Now, applying ($\ast$) with $i = 1$ and $j = 3$, we have that either $x_{1,2} \in \{w_3,w_4\}$ or $x_{3,4} \in \{w_1,w_2\}$.  
Since $n \ge 6$ and $x_{1,2} \in \{w_{n-2},w_{n-1}\}$, we must have 
\[
x_{3,4} \in \{w_1,w_2\}.
\]
Finally, applying ($\ast$) with $i = 3$ and $j = 5$, we have that either $x_{3,4} \in \{w_5,w_6\}$ or $x_{5,6} \in \{w_3,w_4\}$. Since we already have $x_{3,4} \in \{w_1,w_2\}$, we must have 
\[
x_{5,6} \in \{w_3,w_4\}.
\]

Now, by the definition of $x_{i,j}$, we have that $\{u,x_{1,2}\}$, $\{u,x_{3,4}\}$, and $\{u,x_{5,6}\}$ are homogeneous sets of $H_{1,2}$, $H_{3,4}$, and $H_{5,6}$, respectively.
Restricting to the subtournament $H_{1,2,3,4} + \{u,x_{1,2},x_{3,4}\} = H_{1,2,3,4} + \{u,x_{3,4}\}$ (the equality holds because $x_{1,2} \in \{w_{n-2},w_{n-1}\} \subseteq V(H_{1,2,3,4})$), we have that $\{u,x_{1,2}\}$ and $\{u,x_{3,4}\}$ are homogeneous sets of $H_{1,2,3,4} + \{u,x_{3,4}\}$, and hence by Proposition \ref{union} and restriction, we have that $\{x_{1,2},x_{3,4}\}$ is a homogeneous set of $H_{1,2,3,4} + x_{3,4}$. Similarly, $\{x_{3,4},x_{5,6}\}$ is a homogeneous set of $H_{3,4,5,6} + x_{5,6}$. 

Let $i_{1,2}, i_{3,4}, i_{5,6}$ be the integers such that $x_{1,2} = w_{i_{1,2}}$, $x_{3,4} = w_{i_{3,4}}$, and $x_{5,6} = w_{i_{5,6}}$. Then $\{w_{i_{1,2}},w_{i_{3,4}}\}$ is a homogeneous set of $H_{1,2,3,4} + w_{i_{3,4}}$. In particular, we have $d_{vw_{i_{1,2}}} = d_{vw_{i_{3,4}}}$. Hence, by the definition of $v$, $i_{1,2}$ and $i_{3,4}$ have the same parity. Similarly, $i_{3,4}$ and $i_{5,6}$ have the same parity, so $i_{1,2}$, $i_{3,4}$, and $i_{5,6}$ all have the same parity. Suppose these numbers are even. Then $i_{1,2} = n-1$ and $i_{3,4} = 2$. However, we have $w_2 \ra w_{n-2}$ and $w_{n-2} \ra w_{n-1}$ in $H_{1,2,3,4} + w_{i_{3,4}}$, which contradicts the fact that $\{w_{i_{1,2}},w_{i_{3,4}}\} = \{w_2,w_{n-1}\}$ is a homogeneous set of $H_{1,2,3,4} + w_{i_{3,4}}$. Similarly, if the numbers are odd, then $i_{3,4} = 1$ and $i_{5,6} = 3$, but $w_1 \ra w_2$ and $w_2 \ra w_3$ in $H_{3,4,5,6} + w_{i_{5,6}}$, contradicting the fact that $\{w_{i_{3,4}},w_{i_{5,6}}\} = \{w_1,w_3\}$ is a homogeneous set of $H_{3,4,5,6} + w_{i_{5,6}}$. This completes the proof. 

\end{proof}

\subsection{An application to $D_4$-free tournaments} \label{diamondsection}

We conclude this section by using Theorem \ref{grow} to give a simple proof of a structural theorem.
Let $D_4$ be the tournament on 4 vertices consisting of a cyclic triangle $C$ and a vertex $v$ with $v \Ra C$. Let $D_4^\ast$ be the tournament formed from $D_4$ by reversing all of its edges. A well-known theorem classifies all tournaments that exclude both $D_4$ and $D_4^\ast$.

\begin{thm}[Gnanvo and Ille \cite{GI}, Lopez and Rauzy \cite{LR}] \label{diamonds}
A prime tournament $G$ with $|V(G)| \ge 5$ does not have a subtournament isomorphic to $D_4$ or $D_4^\ast$ if and only if $G$ is $T_n$ for some odd $n \ge 5$.
\end{thm}

In fact, this theorem still holds true if we replace excluding both $D_4$ and $D_4^\ast$ with excluding only $D_4$. We give a short proof of this result without relying on Theorem \ref{diamonds} itself.

\begin{thm} \label{diamond}
A prime tournament $G$ with $|V(G)| \ge 3$ does not have a subtournament isomorphic to $D_4$ if and only if $G$ is $T_n$ for some odd $n \ge 3$.
\end{thm}

\begin{proof}
To prove one direction, note that if $G$ is $T_n$ for some odd $n \ge 3$, then the outneighborhood of every vertex $v \in V(G)$ forms a transitive tournament. Since $D_4$ consists of a vertex whose outneighborhood is a cyclic triangle, $G$ does not have a subtournament isomorphic to $D_4$.

We now prove the other direction. Suppose $G$ is a prime tournament with $|V(G)| \ge 3$ that does not have a subtournament isomorphic to $D_4$. If $|V(G)| < 5$, then $G$ is $T_3$ and we are done. So assume $|V(G)| \ge 5$. $G$ cannot be $U_n$ or $W_n$ for any $n$ because these tournaments have subtournaments isomorphic to $D_4$, and if $G$ is $T_n$ for some $n$ then we are done. So assume $G$ is not $T_n$, $U_n$, or $W_n$ for any $n$.

By Corollary \ref{contains5}, $G$ has a prime subtournament $H_5$ with 5 vertices, and since $U_n$ and $W_n$ each have $D_4$ subtournaments, $H_5$ must be $T_5$. If $G = H_5$, we are done; otherwise, by Theorem \ref{grow2}, $G$ has a prime subtournament $H$ with 6 vertices that has a subtournament isomorphic to $T_5$. Let $H^\prime$ be a subtournament of $H$ isomorphic to $T_5$, and let $\{u\} = V(H) \minus V(H^\prime)$. Let the vertices of $H^\prime$ be $v_1,\ldots,v_5$ as in Definition \ref{TUW}.

Since $H$ is prime, we have $u \notin Z(H^\prime)$. Hence, there must be some $1 \le i \le 5$ such that $u \ra v_i$ and $v_{i+1} \ra u$ (where the indices are taken modulo 5). Without loss of generality, assume $u \ra v_2$ and $v_3 \ra u$. Then $\{u,v_2,v_3\}$ forms a cyclic triangle. Since $v_1 \Ra \{v_2,v_3\}$, we must have $u \ra v_1$ (or else $\{v_1,u,v_2,v_3\}$ would form a $D_4$). Now, $\{v_1,v_2,v_4\}$ forms a cyclic triangle, and $u \Ra \{v_1,v_2\}$, so we must have $v_4 \ra u$.
Thus, $\{v_3,v_4\} \Ra u \Ra \{v_1,v_2\}$. But then $\{u,v_5\}$ is a homogeneous set of $H$, which contradicts the primeness of $H$. This completes the proof.
\end{proof}

\begin{cor} \label{diamondgeneral}
A tournament $G$ does not have a subtournament isomorphic to $D_4$ if and only if it can be written as $T_n(I^1,\ldots,I^n)$ or $I_2(T_n(I^1,\ldots,I^n), I)$, where $n \ge 1$ is odd, $I_2$ has the standard ordering of vertices, and $I^1,\ldots,I^n,I$ are transitive tournaments.
\end{cor}

\begin{proof}
If $G$ can be written in one of the forms mentioned, then the outneighborhood of every vertex $v \in V(G)$ forms a transitive tournament, and hence $G$ does not have a subtournament isomorphic to $D_4$. This proves one direction.

Now, suppose $G$ does not have a subtournament isomorphic to $D_4$. If $G$ has only one vertex, then we can write $G$ as $T_1(T_1)$, and we are done. So assume $|V(G)| \ge 2$.
Write $G$ as $G^\prime(H_1,\ldots,H_n)$, where $G^\prime$ is a prime tournament with $|V(G^\prime)| \ge 2$. First, suppose that $|V(G^\prime)| \ge 3$. Then by Theorem \ref{diamonds}, $G^\prime$ is $T_n$ for some odd $n$. Now, if any of $H_1,\ldots,H_n$, say $H_i$, contains a cyclic triangle, then this triangle forms a $D_4$ with any vertex in $H_{i-1}$ (where the index is taken modulo $n$). So $H_i$ is transitive for all $i$, and thus $G$ is of the first form in the theorem, as desired.

Now suppose that $|V(G^\prime)| = 2$, and hence $G^\prime = P_2$. Then $G$ is of the form $I_2(H_1,H_2)$. If $H_2$ contains a cyclic triangle, then this triangle forms a $D_4$ with any vertex in $H_1$; thus, $H_2$ is transitive. Now, write $G$ as $P_2(H_1,I_k)$ in such a way that $k$ is maximal. If $|V(H_1)| = 1$ then $H_1 = I_1$ and we are done. Otherwise, write $H_1$ as $H_1^\prime(K_1,\ldots,K_m)$, where $H_1^\prime$ is prime and $|V(H_1^\prime)| \ge 2$. If $|V(H^\prime)| \ge 3$, then since $H_1$ is $D_4$-free, we have as before that $H_1^\prime$ is $T_m$ for odd $m$ and each $K_i$ is transitive. Thus $G$ is of the second form in the theorem, as desired. Otherwise, $H_1^\prime$ is $P_2$, and we have as before that $H_1$ is $P_2(K_1,I_j)$ for some $j$. But then $G$ can be written as $P_2(K_1,I_{j+k})$, contradicting the maximality of $k$. This completes the proof. 
\end{proof}

\section{Cyclic triangles in prime tournaments}

\subsection{Triangle-connectivity}

Let $G$ be a tournament. Say that two cyclic triangles $C$, $C^\prime$ in $G$ are \emph{adjacent} if they share exactly two vertices. We say that two cyclic triangles $C$, $C^\prime$ are \emph{triangle-connected} to each other if there is a sequence of cyclic triangles $C_1,\ldots,C_n$, $n \ge 1$, such that $C_1 = C$, $C_n = C^\prime$, and $C_i$ is adjacent to $C_{i+1}$ for all $1 \le i \le n-1$. We can ask the following question: when are all the cyclic triangles of a tournament triangle-connected to each other?

Call a tournament \emph{triangle-connected} if it is strongly connected and any two cyclic triangles in the tournament are triangle-connected. In general, nontrivial homogeneous sets that contain cyclic triangles can stop a tournament from being triangle-connected; to see this, let $G$ be a strongly connected tournament and suppose $X$ is a nontrivial homogeneous set of $G$. Then a cyclic triangle all of whose vertices are in $X$ cannot be triangle-connected to a cyclic triangle which has a vertex not in $X$, because if there were indeed a sequence of adjacent cyclic triangles connecting two such triangles, there must be some triangle in the sequence with two vertices $\{v_1,v_2\}$ in $X$ and one vertex $v_3$ not in $X$, a contradiction since $X$ is a homogeneous set so either $v_3 \Ra \{v_1,v_2\}$ or $\{v_1,v_2\} \Ra v_3$.

One could then ask whether a tournament that is strongly connected and which has no nontrivial homogeneous sets that contain cyclic triangles is triangle-connected. The answer is positive, and follows from the next theorem.

\begin{thm} \label{primeconnected}
If $G$ is a strongly connected prime tournament, then it is triangle-connected.
\end{thm}

\begin{cor} \label{connected}
A strongly connected tournament is triangle-connected if and only if it has no nontrivial homogeneous sets that contain cyclic triangles.
\end{cor}

To see that the corollary follows from the theorem, let $G$ be a strongly connected tournament. First, suppose that $X \subseteq V(G)$ is a nontrivial homogeneous set of $G$ which contains a cyclic triangle $C$. Since $X$ is nontrivial, let $v \in V(G) \minus X$. Since $G$ is strongly connected, $v$ is a vertex of some cyclic triangle $C^\prime$ in $G$. Then as noted before, $C$ and $C^\prime$ are not triangle-connected to each other, so $G$ is not triangle-connected. This proves one direction of the corollary.

For the other direction, suppose $G$ has no nontrivial homogeneous sets that contain cyclic triangles. If $G$ has one vertex then it is trivially triangle-connected, so assume $|V(G)| \ge 2$. Write $G$ as $G^\prime(H_1,\ldots,H_n)$ where $G^\prime$ is prime (and strongly connected, since $G$ is). Since each $V(H_i)$ is a homogneneous set of $G$, none of the $H_i$ contain a cyclic triangle. Hence, every cyclic triangle of $G$ has vertices $\{v_i,v_j,v_k\}$ for some distinct $i,j,k$ such that $v_i \in V(H_i)$, $v_j \in V(H_j)$, and $v_k \in V(H_k)$. Call a cyclic triangle a $\{i,j,k\}$-triangle if it has vertices in $V(H_i)$, $V(H_j)$, and $V(H_k)$. It is easy to see that for each set of distinct $i,j,k$, all $\{i,j,k\}$-triangles are triangle-connected to each other. Also, since $G^\prime$ is prime and strongly connected, by Theorem \ref{primeconnected}, every $\{i,j,k\}$-triangle is connected to a $\{i^\prime,j^\prime,k^\prime\}$-triangle for every $\{i^\prime,j^\prime,k^\prime\}$ for which such a triangle exists. Thus, every cyclic triangle of $G$ is triangle-connected to each other, completing the proof of the corollary.

\subsection{Proof of Theorem \ref{primeconnected}}

To prove the theorem, we use Theorem \ref{grow} (or Theorem \ref{critical}) and induction. Because Theorem \ref{grow} does not hold for $T_n$, $U_n$, or $W_n$, we will treat these cases separately.

\begin{proof}[Proof for $T_n$, $U_n$, and $W_n$]
Let $G$ be $T_n$, $U_n$, or $W_n$ for some odd $n \ge 1$. We prove that $G$ is triangle-connected. The case $n = 1$ is trivial, so assume $n \ge 3$. As in the proof of Theorem \ref{grow}, write $G$ as $W+v$, where $W = \langle w_1, \ldots, w_{n-1} \rangle$ is a weave and $v$ is as in Proposition \ref{weave}.

It suffices to prove that every cyclic triangle in $G$ is triangle-connected to the cyclic triangle with vertices $\{v,w_1,w_{n-1}\}$. Given distinct vertices $x,y,z \in V(G)$, let $C(x,y,z)$ be the triangle with vertices $\{x,y,z\}$ (the triangle may or may not be cyclic). Every cyclic triangle in $G$ is of one of the following forms:

\renewcommand{\labelenumi}{(\alph{enumi})}
\begin{enumerate}
\item $C(v,w_i,w_j)$, where $i < j$ and $i$ is odd and $j$ is even. (All such triangles are cyclic.)
\item $C(w_i,w_j,w_k)$ where $i < j < k$ and $j$ has opposite parity from $i$ and $k$. (Such a triangle might not be cyclic depending on the type of weave.)
\end{enumerate}

Let $C$ be a cyclic triangle in $G$. If $C = C(v,w_i,w_j)$ as in (a), we have that $C$ is adjacent to or the same as the cyclic triangle $C(v,w_1,w_j)$, and this triangle is adjacent to or the same as $C(v,w_1,w_{n-1})$. Thus, $C$ is triangle-connected to $C(v,w_1,w_{n-1})$, as desired.

Now suppose $C = C(w_i,w_j,w_k)$ as in (b). Then one of the triangles $C(v,w_i,w_j)$, $C(v,w_j,w_k)$ is of type (a), and hence this triangle is cyclic and, as above, it is triangle-connected to $C(v,w_1,w_{n-1})$. Since $C$ is adjacent to this cyclic triangle, $C$ is also triangle-connected to $C(v,w_1,w_{n-1})$, as desired. This completes the proof for $T_n$, $U_n$, and $W_n$.

\end{proof}

Now, let $G$ be strongly connected and prime. We will prove that $G$ is triangle-connected by induction on $|V(G)|$. All base cases $|V(G)| \le 5$ were covered in the proof for $T_n$, $U_n$, and $W_n$. Suppose $|V(G)| \ge 6$, and that the thoerem holds for tournaments with $|V(G)| - 1$ vertices. If $G$ is $T_n$, $U_n$, or $W_n$ for any $n$, then from the above argument we are done. Assume $G$ is not one of these tournaments. By Theorem \ref{grow} (or Theorem \ref{critical}), there is a prime subtournament $H$ of $G$ with $|V(G)| - 1$ vertices. Let $\{u\} = V(G) \minus V(H)$. Since $H$ is triangle-connected by the inductive hypothesis, to show that $G$ is triangle-connected it suffices to show that every cyclic triangle of $G$ with $u$ as a vertex is triangle-connected to a cyclic triangle in $H$.

Suppose the contrary. Let $\mc{C}$ denote the set of cyclic triangles of $G$ that have $u$ as a vertex. Let $C = C(u,v_A,v_B)$ be a cyclic triangle in $\mc{C}$ that is not triangle-connected to a cyclic triangle in $H$, where $u \ra v_A \ra v_B \ra u$. We will prove the following two claims:

\setcounter{claim}{0}
\begin{claim}
$C$ is triangle-connected to every cyclic triangle in $\mc{C}$.
\end{claim}

\begin{claim}
There is a cyclic triangle in $\mc{C}$ which is triangle-connected to a cyclic triangle in $H$.
\end{claim}

These two claims clearly contradict the definition of $C$, which will complete the proof.

\begin{proof}[Proof of Claim 1]
This proof is due to P. Seymour. Let $A = A_G(u)$ and $B = B_G(u)$, so $v_A \in A$, $v_B \in B$, and $A \cup B = V(H)$. We associate each cyclic triangle $C(u,x_A,x_B)$ in $\mc{C}$, where $x_A \in A$ and $x_B \in B$, with the directed edge $x_Ax_B$; this gives a one-to-one correspondence between the triangles in $\mc{C}$ and the edges from $A$ to $B$. Let $\hat{H}$ be the undirected graph with vertices $V(H)$ and an undirected edge $xy$ for each directed edge $xy$ in $H$ with $x \in A$ and $y \in B$. Then two triangles in $\mc{C}$ are triangle-connected if and only if their associated edges in $\hat{H}$ are in the same connected component of $\hat{H}$.

Let $\hat{H}^\prime$ be the connected component of $\hat{H}$ containing $v_Av_B$, and let $A^\prime = A \cap V(\hat{H}^\prime)$ and $B^\prime = B \cap V(\hat{H}^\prime)$. Suppose there is a vertex $v \in A$ which is not in $A^\prime$. We claim that $\{u\} \cup A^\prime \cup B^\prime \Ra v$. Because $v$ is not adjacent in $\hat{H}$ to any vertex in $B^\prime$, we must have $B^\prime \Ra v$ in the tournament $H$. Moreover, if there is a vertex $x_{A^\prime} \in A^\prime$ such that $v \ra x_{A^\prime}$, then $x_{A^\prime} \ra x_{B^\prime}$ for some $x_{B^\prime} \in B^\prime$, and $C(v,x_{A^\prime},x_{B^\prime})$ is a cyclic triangle in $H$. $C$ is triangle-connected to this triangle, which contradicts the fact that $C$ is not triangle-connected to a triangle in $H$. Thus, we must have $A^\prime \Ra v$. Finally, $u \ra v$ by the definition of $A$, so altogether we have $\{u\} \cup A^\prime \cup B^\prime \Ra v$, as claimed. Similarly, if $v \in B$ and $v \notin B^\prime$, we have $v \Ra \{u\} \cup A^\prime \cup B^\prime$. It follows that $\{u\} \cup A^\prime \cup B^\prime$ is a homogeneous set of $H$. Since this set contains $\{u,v_A,v_B\}$ and $G$ is prime, we must have $\{u\} \cup A^\prime \cup B^\prime = V(G)$, and hence the connected component $\hat{H}^\prime$ is all of $\hat{H}$. Thus, $C$ is triangle-connected to every triangle in $\mc{C}$, as desired.
\end{proof}

\begin{proof}[Proof of Claim 2]
Suppose that no cyclic triangle in $\mc{C}$ is triangle-connected to a cyclic triangle in $H$. Since every four-vertex tournament is decomposable, for every cyclic triangle $C(x,y,z)$ in $H$, we have that either $u \in Z(C(x,y,z))$ or $u \in V_v(C(x,y,z))$ for some $v \in \{x,y,z\}$. If the latter case, say $u \in V_x(C(x,y,z))$, then $C(u,y,z)$ is a cyclic triangle, and it is adjacent to $C(x,y,z)$, contradicting our original assumption. Thus, $u \in Z(C^\prime)$ for every cyclic triangle $C^\prime$ in $H$, and hence for every such cyclic triangle either $u \Ra C^\prime$ or $C^\prime \Ra u$. However, by the inductive hypothesis, any two cyclic triangles in $H$ can be connected to each other by a sequence of adjacent cyclic triangles, so in fact we have that either $u \Ra C^\prime$ for every cyclic triangle $C^\prime$ in $H$ or $C^\prime \Ra u$ for every cyclic triangle $C^\prime$ in $H$. Since $H$ is strongly connected, every vertex of $H$ belongs to a cyclic triangle of $H$. Thus, $V(H)$ is a homogeneous set of $G$, contradicting the fact that $G$ is prime. This completes the proof.
\end{proof}

\section{Matching tournaments}

\subsection{Matching orderings} \label{matchingorderings}

We now focus on a somewhat different topic. Given a tournament $G$ and an ordering $v_1,\ldots,v_n$ of its vertices, an edge $v_jv_i$ with $j > i$ is called a \emph{backedge}. A \emph{matching ordering} of a tournament $G$ is an ordering of its vertices such that every vertex is the head or tail of at most one backedge. A tournament with at least one matching ordering is called a \emph{matching tournament}. In Subsections 5.1 through 5.4, we will deal with the question of how many matching orderings a matching tournament can have. Subsection 5.5 will deal with minimal non-matching tournaments.

In general, a matching tournament can have many matching orderings. For example, $I_n$ has at least $2^{n/2}$ matching orderings: Let $v_1,\ldots,v_{n}$ be the standard ordering of $V(I_n)$. Let $S_n$ denote the symmetric group on $\{1,2,\ldots,n\}$, and let $\tau_i$ denote the transposition $(i \ps i+1)$. Then $v_{\pi(1)},\ldots,v_{\pi(n)}$ is a matching ordering of $I$ for all permutations $\pi \in S_{n}$ of the form $\pi = \tau_1^{e_1}\tau_3^{e_3}\tau_5^{e_5}\cdots\tau_{2\lfloor n/2 \rfloor - 1}^{e_{2\lfloor n/2 \rfloor - 1}}$ and $\pi = \tau_2^{e_2}\tau_4^{e_4}\tau_6^{e_6}\cdots\tau_{2\lfloor n/2 \rfloor - 2}^{e_{2\lfloor n/2 \rfloor - 2}}$, where $e_i = 0$ or 1 for all $i$.

An example of an infinite family of prime tournaments with more than one matching ordering is as follows: For $n \ge 1$, let $P_n$ be the tournament with vertices $v_1, \ldots, v_n$ such that $v_i \ra v_j$ if $j - i \ge 2$, and $v_{i+1} \ra v_i$ for all $1 \le i \le n-1$. For all $n \neq 4$, $P_n$ is prime. (Note that $P_2$ is $I_2$, $P_3$ is the cyclic triangle, and $P_5$ is $W_5$.) Let $\pi_1 = \tau_1 \tau_3 \ldots \tau_{2\lfloor n/2 \rfloor - 1}$ and $\pi_2 = \tau_2 \tau_4 \ldots \tau_{2\lfloor n/2 \rfloor - 2}$, where $\tau_i = (i \ps i+1)$ as before. Then $v_{\pi_1(1)},\ldots,v_{\pi_1(n)}$ and $v_{\pi_2(1)},\ldots,v_{\pi_2(n)}$ are both matching orderings of $P_n$, and these orderings are distinct if $n > 1$. Below we draw these two matching orderings for both odd and even examples of $n$; only backedges are shown.

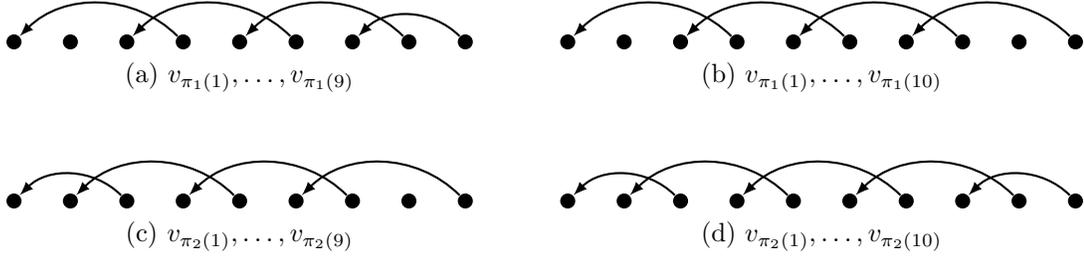
\begin{figure}[htbp]
\centering
\subfloat[$v_{\pi_1(1)},\ldots,v_{\pi_1(9)}$]{
\begin{tikzpicture}[scale=0.75]
\node[vertex] (v1) at (1,0) {};
\node[vertex] (v2) at (2,0) {};
\node[vertex] (v3) at (3,0) {};
\node[vertex] (v4) at (4,0) {};
\node[vertex] (v5) at (5,0) {};
\node[vertex] (v6) at (6,0) {};
\node[vertex] (v7) at (7,0) {};
\node[vertex] (v8) at (8,0) {};
\node[vertex] (v9) at (9,0) {};
\draw[arc] (v4) to (v1);
\draw[arc] (v6) to (v3);
\draw[arc] (v8) to (v5);
\draw[arc] (v9) to (v7);
\end{tikzpicture}
} \hspace{.75cm}
\subfloat[$v_{\pi_1(1)},\ldots,v_{\pi_1(10)}$]{
\begin{tikzpicture}[scale=0.75]
\node[vertex] (v1) at (1,0) {};
\node[vertex] (v2) at (2,0) {};
\node[vertex] (v3) at (3,0) {};
\node[vertex] (v4) at (4,0) {};
\node[vertex] (v5) at (5,0) {};
\node[vertex] (v6) at (6,0) {};
\node[vertex] (v7) at (7,0) {};
\node[vertex] (v8) at (8,0) {};
\node[vertex] (v9) at (9,0) {};
\node[vertex] (v10) at (10,0) {};
\draw[arc] (v4) to (v1);
\draw[arc] (v6) to (v3);
\draw[arc] (v8) to (v5);
\draw[arc] (v10) to (v7);
\end{tikzpicture}
} \vspace{0.75cm}
\subfloat[$v_{\pi_2(1)},\ldots,v_{\pi_2(9)}$]{
\begin{tikzpicture}[scale=0.75]
\node[vertex] (v1) at (1,0) {};
\node[vertex] (v2) at (2,0) {};
\node[vertex] (v3) at (3,0) {};
\node[vertex] (v4) at (4,0) {};
\node[vertex] (v5) at (5,0) {};
\node[vertex] (v6) at (6,0) {};
\node[vertex] (v7) at (7,0) {};
\node[vertex] (v8) at (8,0) {};
\node[vertex] (v9) at (9,0) {};
\draw[arc] (v3) to (v1);
\draw[arc] (v5) to (v2);
\draw[arc] (v7) to (v4);
\draw[arc] (v9) to (v6);
\end{tikzpicture}
} \hspace{.75cm}
\subfloat[$v_{\pi_2(1)},\ldots,v_{\pi_2(10)}$]{
\begin{tikzpicture}[scale=0.75]
\node[vertex] (v1) at (1,0) {};
\node[vertex] (v2) at (2,0) {};
\node[vertex] (v3) at (3,0) {};
\node[vertex] (v4) at (4,0) {};
\node[vertex] (v5) at (5,0) {};
\node[vertex] (v6) at (6,0) {};
\node[vertex] (v7) at (7,0) {};
\node[vertex] (v8) at (8,0) {};
\node[vertex] (v9) at (9,0) {};
\node[vertex] (v10) at (10,0) {};
\draw[arc] (v3) to (v1);
\draw[arc] (v5) to (v2);
\draw[arc] (v7) to (v4);
\draw[arc] (v9) to (v6);
\draw[arc] (v10) to (v8);
\end{tikzpicture}
}
\end{figure}

Despite these examples, having more than one matching ordering is in fact a very strict condition. One reason for this is a simple consideration of vertex degrees, which gives the following fact.

\begin{prop} \label{degrees}
Let $v_1,\ldots,v_n$ be a matching ordering of $G$. If $v$ is a vertex of $G$ with indegree $b$, then $v$ is either $v_b$, $v_{b+1}$, or $v_{b+2}$. Moreover, $v = v_b$ if and only if it is the head of a backedge, $v = v_{b+2}$ if and only if it is the tail of a backedge, and $v = v_{b+1}$ if and only if it is not an end of a backedge.
\end{prop}

\begin{proof}
Since $v_i \ra v_j$ for each $i < j$ except for backedges $v_jv_i$, and every vertex is the end of at most one backedge, we have $|B_G(v_i)| = i$ if $v_i$ is the head of a backedge, $|B_G(v_i)| = i-2$ if $v_i$ is the tail of a backedge, and $|B_G(v_i)| = i-1$ if $v_i$ is not an end of a backedge.
\end{proof}

We will show that, in a sense, $P_n$ is the only ``fundamental'' example of a tournament with more than one matching ordering. More specifically, we will show that the only reason a tournament might have more than one matching ordering is that it has a homogeneous set on which the induced subtournament is $P_n$ for $n > 1$.

\subsection{Statement and proof of theorem}

Before stating the exact theorem, we introduce several definitions and facts. 
Let $X = \{b, b+1, \ldots, a\}$ be a nonempty set of consecutive integers. Define $\sigma_X$ to be the permutation on $X$ as follows.
\begin{itemize}
\item If $|X|$ is odd, $\sigma_X = (b \ps b+2 \ps b+4 \ps \ldots \ps a-2 \ps a \ps a-1 \ps a-3 \ps \ldots \ps b+1)$.
\item If $|X|$ is even, $\sigma_X = (b \ps b+2 \ps b+4 \ps \ldots \ps a-1 \ps a \ps a-2 \ps a-4 \ps \ldots \ps b+1)$.
\end{itemize}
(If $|X| = 1$, $\sigma_X$ is the identity, and if $|X| = 2$, $\sigma_X$ is a transposition.) 
Defining $\pi_1 = \tau_1 \tau_3 \ldots \tau_{2\lfloor n/2 \rfloor - 1}$ and $\pi_2 = \tau_2 \tau_4 \ldots \tau_{2\lfloor n/2 \rfloor - 2}$ as in the previous subsection, we also have that $\sigma_{\{1,\ldots,n\}} = \pi_1\pi_2$. 

If $X$ is as above and $|X| = 4$, define $\tau_X = (b \ps b+2) (b+1 \ps b+3)$. We have $\tau_{\{1,2,3,4\}} = \pi_2 \pi_1 \pi_2$, where we use $n=4$ in the definition of $\pi_1$ and $\pi_2$.

Finally, suppose $\pi \in S_n$. Given a transitive tournament $G$ with the standard ordering $v_1,\ldots,v_n$ of $V(G)$, let $\pi G$ denote the ordering $v_{\pi(1)},\ldots,v_{\pi(n)}$ of $V(G)$. Similarly, given a tournament $G$ isomorphic to $P_n$, if $n \neq 3$ then there is a unique ordering $v_1,\ldots,v_n$ of its vertices which satisfies the definition of $P_n$ given previously. Let $\pi G$ denote the ordering $v_{\pi(1)},\ldots,v_{\pi(n)}$ of $V(G)$. If $G$ is $P_3$, then there is no unique ordering satisfying the definition, so we will say that an ordering is $\pi G$ if it is $v_{\pi(1)},v_{\pi(2)},v_{\pi(3)}$ for \emph{some} ordering $v_1,v_2,v_3$ of $V(G)$ that satisfies the definition of $P_n$. If $|V(G)| = 2$, then $G$ is isomorphic to both $I_2$ and $P_2$ but the two above definitions give different orderings for $\pi G$, so in this case we let $\pi G$ denote any ordering of $V(G)$.

Note that $|\sigma_X(x) - x| \le 2$ and $|\tau_X(x) - x| \le 2$ for all $x \in X$. Also, as we noted in the previous subsection, if $G$ is transitive or isomorphic to $P_n$, then $\pi_1 G$ and $\pi_2 G$ are matching orderings of $G$. If $G$ is $P_4$, we can check that $\pi_2 \pi_1 G$ is also a matching ordering of $G$. The next two propositions can be thought of as converses to these facts.

\begin{prop} \label{Xcycle}
Let $\sigma$ be a permutation of $\{1,2,\ldots,n\}$ which can be written as a cycle $(x_1 \ps x_2 \ps \ldots \ps x_k)$, $k \ge 1$. If $|\sigma(x) - x| \le 2$ for all $x \in \{1,\ldots,n\}$, then either 
\begin{itemize}
\item $X = \{x_1,\ldots, x_k\}$ is a set of consecutive integers and $\sigma = \sigma_X$ or $\sigma_X^{-1}$, or
\item $k=2$ and $|x_1 - x_2| = 2$.
\end{itemize}
\end{prop}

\begin{prop} \label{Xorderings}
Let $G$ be a tournament with matching ordering $v_1,\ldots,v_n$, and let $X = \{1,\ldots,n\}$. Suppose $\sigma \in S_n$ such that $v_{\sigma(1)},\ldots,v_{\sigma(n)}$ is a matching ordering of $G$. Then the following hold.
\begin{itemize}
\item If $\sigma = \sigma_X$, then $G$ is either transitive or isomorphic to $P_n$, and the ordering $v_1,\ldots,v_n$ is $\pi_2 G$.
\item If $\sigma = \sigma_X^{-1}$, then $G$ is either transitive or isomorphic to $P_n$, and the ordering $v_1,\ldots,v_n$ is $\pi_1 G$.
\item If $n = 4$ and $\sigma = \tau_X$, then $G$ is isomorphic to $P_4$, and the ordering $v_1,\ldots,v_n$ is either $\pi_2 G$ or $\pi_2 \pi_1 G$.
\end{itemize}
\end{prop}

For the sake of pacing, we defer the proofs of these propositions until after the next subsection.

We now state and prove the main theorem.

\begin{thm} \label{orderings}
Let $G$ be a tournament and $v_1,\ldots,v_n$ be a matching ordering of $G$. Suppose there is a permutation $\pi \in S_n$ such that $v_{\pi(1)}, \ldots, v_{\pi(n)}$ is also a matching ordering of $G$. Then $\{1,\ldots,n\}$ can be partitioned into sets $X_1,\ldots,X_r$, where each $X_i$ is a set of consecutive integers, such that
\begin{itemize}
\item $\pi = \sigma_1 \sigma_2 \ldots \sigma_r$ where for each $1 \le i \le r$, $\sigma_i$ is either $\sigma_{X_i}$, $\sigma_{X_i}^{-1}$, or $\tau_{X_i}$.
\item For each $1 \le i \le r$, the set $V_i = \{v_x : x \in X_i\}$ is a homogeneous set of $G$.
\item For each $1 \le i \le r$, the subtournament $H_i$ induced on $V_i$ is either transitive or isomorphic to $P_{|V_i|}$. Furthermore, the ordering induced on $V_i$ by the ordering $v_1,\ldots,v_n$ is given by $\sigma_i$ as in Proposition \ref{Xorderings}.
\end{itemize}
\end{thm}

\begin{proof}
Suppose vertex $v_{\pi(i)}$ has indegree $b$. By Proposition \ref{degrees}, for any matching ordering $u_1,\ldots,u_n$ of $G$, we have that $v_{\pi(i)}$ appears in the $b$-th, $(b+1)$-th, or $(b+2)$-th position. In particular, since $v_{\pi(i)}$ appears in the $\pi(i)$-th position of $v_1,\ldots,v_n$ and in the $i$-th position of $v_{\pi(1)}, \ldots, v_{\pi(n)}$, we have $\pi(i),i \in \{b,b+1,b+2\}$. It follows that $|\pi(i) - i| \le 2$ for all $1 \le i \le n$. 

Now, we can write $\pi = \sigma_1^\prime \cdots \sigma_{r^\prime}^\prime$ where $\sigma_1^\prime, \ldots, \sigma_{r^\prime}^\prime$ are disjoint cycles whose orbits form a partition of $\{1,\ldots,n\}$. (Some of the $\sigma_i^\prime$ may have orbits of size one.) Let $X_i^\prime$ be the orbit of $\sigma_i^\prime$. Now, for each $1 \le i \le r^\prime$, we have $|\sigma_i^\prime(x) - x| \le 2$ for all $1 \le x \le n$. Hence, by Proposition \ref{Xcycle}, for each $1 \le i \le r^\prime$ either $X_i^\prime$ is a set of consecutive integers and $\sigma_i^\prime = \sigma_{X_i}^{\pm 1}$, or $X_i^\prime = \{x, x+2\}$ for some $x$.

Suppose $X_i^\prime = \{x, x+2\}$ for some $i$. The integer $x+1$ must belong to $X_j^\prime$ for some $j \neq i$. Either $X_j^\prime$ is a set of consecutive integers, or $X_j^\prime = \{x-1,x+1\}$ or $\{x+1,x+3\}$. Suppose the former case; then since $x,x+2 \in X_i^\prime$, we must have $X_j^\prime = \{x+1\}$. Thus, $\pi(x+1) = x+1$. Now, consider the subtournament $H$ induced on $\{v_x,v_{x+1},v_{x+2}\}$. Using the orderings of $V(H)$ induced by the matching orderings $v_1,\ldots,v_n$ and $v_{\pi(1)},\ldots,v_{\pi(n)}$, we have that $v_x,v_{x+1},v_{x+2}$ and $v_{x+2},v_x,v_{x+1}$ are both matching orderings of $H$. However, there are only two three-vertex tournaments, and simple inspection shows that neither has two matching orderings satisfying this. This is a contradiction, so $X_j^\prime$ is not a set of consecutive integers.

Hence, $X_j^\prime = \{x-1,x+1\}$ or $\{x+1,x+3\}$. Either way, $X_i^\prime \cup X_j^\prime$ is a set of four consecutive integers, and $\sigma_i^\prime \sigma_j^\prime = \tau_{X_i^\prime \cup X_j^\prime}$. We can pair up all the $X_i^\prime$ of the form $\{x,x+2\}$ in this manner. Replacing each such pair with their union and including all the $X_i^\prime$ that were originally sets of consecutive integers, we have in the end a partition $\{X_1,\ldots,X_r\}$ of $\{1,\ldots,n\}$ where each $X_i$ is a set of consecutive integers, and $\pi = \sigma_1 \cdots \sigma_r$ where $\sigma_i$ is either $\sigma_{X_i}^{\pm1}$ or $\tau_{X_i}$, as desired.

Next, we wish to show that each $V_i = \{v_x : x \in X_i\}$ is a homogeneous set in $G$. If $|V_i| = 1$ then we are trivially done. So assume $|V_i| \ge 2$. Let $X_i = \{b, b+1, \ldots, a\}$. We will show that if $b^\prime < b$ then $v_{b^\prime} \Ra V_i$, and if $a^\prime > a$ then $V_i \Ra v_{a^\prime}$, which will prove the claim. Let $b^\prime < b$, and suppose there is some $x \in X_i$ such that $v_x \ra v_{b^\prime}$. Then $v_x$ is the tail of the backedge $v_x v_{b^\prime}$ in the ordering $v_1,\ldots,v_n$. Now, since $b^\prime$ and $x$ are in different cycle orbits of $\pi$, and these orbits are sets of consecutive integers, $v_x v_{b^\prime}$ is also a backedge of $v_{\pi(1)}, \ldots, v_{\pi(n)}$. Thus, $v_x$ is also the tail of the backedge $v_x v_{b^\prime}$ in $v_{\pi(1)}, \ldots, v_{\pi(n)}$. By the second sentence of Proposition \ref{degrees}, $v_x$ must therefore be in the same position in $v_{\pi(1)}, \ldots, v_{\pi(n)}$ as it is in $v_1,\ldots,v_n$. However, this is a contradiction, because $x \in X_i$ and $\sigma_i$ has no fixed points in $X_i$ for $|X_i| \ge 2$. We therefore have $v_{b^\prime} \Ra V_i$. The proof that $V_i \Ra v_{a^\prime}$ for $a^\prime > a$ is analagous, so we have the desired claim.

Finally, to prove the last point, let $X_i = \{b, b+1, \ldots, a\}$. The orderings induced on $V_i$ by $v_1,\ldots,v_n$ and $v_{\pi(1)}, \ldots, v_{\pi(n)}$ are $v_b, v_{b+1}, \ldots, v_a$ and $v_{\sigma_i(b)}, v_{\sigma_i(b+1)}, \ldots, v_{\sigma_i(a)}$, respectively. These are matching orderings of $H_i$, so by Proposition \ref{Xorderings}, $H_i$ is either transitive or isomorphic to $P_{|V_i|}$. Furthermore, the ordering $v_b, v_{b+1}, \ldots, v_a$ is given by $\sigma_i$ as in Proposition \ref{Xorderings}, as desired.
\end{proof}

\subsection{Corollaries to Theorem \ref{orderings}}

We state a few notable corollaries to the previous theorem.

\begin{cor} \label{primematching}
Every prime tournament which is not $P_n$ for any $n$ has at most one matching ordering.
\end{cor}

\begin{proof}
Let $G$ be a prime tournament which is not $P_n$ for any $n$, and suppose $v_1,\ldots,v_n$ is a matching ordering of $G$. Let $\pi \in S_n$ such that $v_{\pi(1)},\ldots,v_{\pi(n)}$ is a matching ordering of $G$. Write $\pi = \sigma_1\sigma_2\cdots\sigma_r$ as in Theorem \ref{orderings}. If $1 < |X_i| < |V(G)|$ for any $i$, then $V_i = \{v_x : x \in X_i\}$ is a nontrivial homogeneous set of $G$, a contradiction. So we must either have $|X_i| = 1$ for all $i$, or $r=1$ and $X_1 = \{1,\ldots,n\}$. If the former case, then $\pi$ is the identity, so $v_{\pi(1)},\ldots,v_{\pi(n)}$ and $v_1,\ldots,v_n$ are the same ordering. If the latter case, then $V_1 = V(G)$ and so by the third point of Theorem \ref{orderings}, $G$ is either transitive or $P_n$, contradicting our assumptions on $G$. Thus, $G$ has at most one matching ordering.
\end{proof}

\begin{cor} \label{Pnmatching}
$P_n$ has exactly one matching ordering for $n = 1$, exactly three matching orderings for $n=3,4$, and exactly two matching orderings for all other $n$.
\end{cor}

\begin{proof}
We can check the cases $n = 1,2,3,4$ by hand. Assume $n \ge 5$; thus, $P_n$ is prime. Let $v_1,\ldots,v_n$ be a matching ordering of $P_n$, and suppose $\pi \in S_n$ such that $v_{\pi(1)},\ldots,v_{\pi(n)}$ is a matching ordering of $P_n$. By the same argument as in the previous proof, either $\pi = 1$ or $\pi = \sigma_X^e$, where $X = \{1,\ldots,n\}$ and $e = \pm 1$. (We do not have $\pi = \tau_X$ because $|X| \ge 5$.) If $\pi = \sigma_X^e$, then there is only one possible value of $e$, because by Proposition \ref{Xorderings}, $e$ is determined by whether $v_1,\ldots,v_n$ is $\pi_1 P_n$ or $\pi_2 P_n$. Thus, either $\pi=1$ or $\pi = \sigma_X^e$ where there is only one possible value of $e$. It follows that $P_n$ has at most two matching orderings. Since $\pi_1 P_n$ are $\pi_2 P_n$ are distinct matching orderings of $P_n$, $P_n$ has exactly two matching orderings, as desired.
(Note: Although $P_3$ is prime, the above argument does not work for $P_3$ because the way we have defined $\pi P_3$, $v_1,v_2,v_3$ can be both $\pi_1 P_3$ and $\pi_2 P_3$, and hence we cannot determine $e$ from Proposition \ref{Xorderings}.)
\end{proof}

\begin{cor} \label{transitivematching}
$I_n$ has exactly $F_n$ matching orderings, where $\{F_i\}_{i \ge 0}$ are the Fibonacci numbers defined by $F_0 = F_1 = 1$ and $F_i = F_{i-1} + F_{i-2}$ for $i \ge 2$.
\end{cor}

\begin{proof}
Let $v_1,\ldots,v_n$ be the standard ordering of $I_n$. Suppose $\pi \in S_n$ such that $v_{\pi(1)},\ldots,v_{\pi(n)}$ is a matching ordering of $I_n$. For this $\pi$, let $X_1,\ldots, X_r$ be as in Theorem \ref{orderings}. We claim that $|X_i| \le 2$ for all $i$.
By Proposition \ref{Xorderings}, for each $i$ the ordering induced on $V_i = \{v_x : x \in X_i\}$ by $v_1,\ldots,v_n$ is of the form $\pi_s I_{|V_i|}$, $\pi_s P_{|V_i|}$, or $\pi_2\pi_1 P_4$, where $s = 0$ or 1. However, $v_1,\ldots,v_n$ has no backedges, and by inspection the only orderings that are of one of these forms and have no backedges are $\pi_s I_1$ (or $\pi_s P_1$) and $\pi_2 I_2$ (or $\pi_1 P_2$). Thus, $|X_i| \le 2$ for all $i$, as claimed. 

Conversely, for each partition $\{X_1,\ldots,X_r\}$ of $\{1,\ldots,n\}$ where each $X_i$ is a set of consecutive integers of size 1 or 2, there is exactly one $\pi = \sigma_1\cdots\sigma_r$ satisfying Theorem \ref{orderings} which has $\{X_1,\ldots,X_r\}$ as its associated partition, and it is not hard to see that $v_{\pi(1)},\ldots,v_{\pi(n)}$ is a matching ordering of $I$ for this $\pi$. Hence, the number of matching orderings of $I$ is equal to the number of ways to partition $\{1,\ldots,n\}$ into sets of consecutive integers of size 1 or 2. We can prove by induction that this number is $F_n$.
\end{proof}

\subsection{Proofs of Propositions \ref{Xcycle} and \ref{Xorderings}}

We now give the proofs of Propositions \ref{Xcycle} and \ref{Xorderings}.

\begin{proof}[Proof of Proposition \ref{Xcycle}]
It is easy to see the proposition holds for $k=1,2$. Assume $k \ge 3$.
Let $x = \min(x_1,x_2,\ldots,x_k)$. We have both $|\sigma(x_i) - x_i| \le 2$ and $|\sigma^{-1}(x_i) - x_i| \le 2$ for all $1 \le i \le k$. In particular, since $x = \min(x_1,,\ldots,x_k)$, we have $\sigma(x),\sigma^{-1}(x) \in \{x+1,x+2\}$. Since the order of $\sigma$ is $k \ge 3$, we have $\sigma(x) \neq \sigma^{-1}(x)$. Hence, $\{\sigma(x),\sigma^{-1}(x)\} = \{x+1,x+2\}$.

Suppose $\sigma^{-1}(x) = x+1$, so $\sigma(x) = x+2$. We will prove by induction that for all $0 \le j \le k-1$, 
\[
\sigma^{(-1)^j \lceil j/2 \rceil}(x) = x+j.
\]
This gives the value of $\sigma^i(x)$ for $k$ consecutive values of $i$, and these values coincide with $\sigma_X^i(x)$, where $X = \{x,x+1,\ldots,x+k-1\}$. Since $\sigma$ and $\sigma_X$ are both cycles of order $k$, this will imply $\sigma = \sigma_X$, as desired.

The base case $j = 0$ is trivial, and the base cases $j = 1,2$ are given by assumption. Let $3 \le j \le k-1$, and suppose that $\sigma^{(-1)^{j_0} \lceil j_0/2 \rceil}(x) = x+j_0$ for all $0 \le j_0 \le j-1$. In particular, for $j_0 = j-2$ we have
\begin{align*}
\sigma^{(-1)^{j-2} \lceil (j-2)/2 \rceil}(x) &= x+j-2 \\
\sigma^{(-1)^j (\lceil j/2 \rceil - 1)}(x) &= x+j-2 \\
\Ra \sigma^{(-1)^j \lceil j/2 \rceil}(x) &= \sigma^{(-1)^j}(x+j-2).
\end{align*}
Now, $|\sigma^{(-1)^j}(x+j-2) - (x+j-2)| \le 2$, so in particular, $\sigma^{(-1)^j}(x+j-2) \le x+j$. Combined with the previous equality, we have
\[
\sigma^{(-1)^j \lceil j/2 \rceil}(x) \le x+j.
\]
If $\sigma^{(-1)^j \lceil j/2 \rceil}(x) < x+j$, then by the inductive hypothesis we have
\[
\sigma^{(-1)^j \lceil j/2 \rceil}(x) = \sigma^{(-1)^{j_0} \lceil {j_0}/2 \rceil}(x)
\]
for some $0 \le j_0 \le j-1$. However, then we have
\[
\sigma^{(-1)^j \lceil j/2 \rceil - (-1)^{j_0} \lceil {j_0}/2 \rceil}(x) = x,
\]
which is a contradiction because the order of $\sigma$ is $k$, and
\begin{align*}
\left| (-1)^j \lceil j/2 \rceil - (-1)^{j_0} \lceil {j_0}/2 \rceil \right| &\le \lceil j/2 \rceil + \lceil {j_0}/2 \rceil \\
&\le \lceil (k-1)/2 \rceil + \lceil (k-2)/2 \rceil \\
&< k.
\end{align*}
Hence, we must have $\sigma^{(-1)^j \lceil j/2 \rceil}(x) = x+j$, which completes the induction and the proof if $\sigma^{-1}(x) = x+1$.

If instead $\sigma(x) = x+1$, apply the above argument to $\sigma^{-1}$. Then $\sigma^{-1} = \sigma_X$, where $X = \{x,x+1,\ldots,x+k-1\}$, so $\sigma = \sigma_X^{-1}$, as desired.

\end{proof}

\begin{proof}[Proof of Proposition \ref{Xorderings}]
We will use the following lemma.

\begin{lem} \label{Xorderingslem}
Let $G$ be a tournament, and let $v_1,\ldots,v_n$ be an ordering of $V(G)$ (not necessarily a matching ordering). Suppose that $v_{\pi_1(1)},\ldots,v_{\pi_1(n)}$ and $v_{\pi_2(1)},\ldots,v_{\pi_2(n)}$ are both matching orderings of $G$. Then $G$ is either transitive or isomorphic to $P_n$, and $v_1,\ldots,v_n$ is the ordering $1 \cdot G$, where 1 is the identity element of $S_n$.
\end{lem}

\begin{proof}
We first introduce some notation: for distinct vertices $u,v$, let $e(u,v)$ denote the edge of $G$ with both ends in $\{u,v\}$. For an integer $i$, let $\pi_i = \pi_1$ if $i$ is odd and $\pi_i = \pi_2$ if $i$ is even. 

Now, note that an edge $e(u,v)$ is a backedge of $v_{\pi_i(1)},\ldots,v_{\pi_i(n)}$ if and only if either
\begin{itemize}
\item $e(u,v)$ is a backedge of $v_1,\ldots,v_n$ and $\{u,v\} \neq \{j,j+1\}$ for any $j \equiv i \pmod{2}$, or
\item $e(u,v)$ is not a backedge of $v_1,\ldots,v_n$ and $\{u,v\} = \{j,j+1\}$ for some $j \equiv i \pmod{2}$.
\end{itemize}
We refer to this fact as ($\ast$). 

We first prove that $v_1,\ldots,v_n$ has no backedge $v_jv_i$ with $j-i \ge 2$.
Suppose that $v_1,\ldots,v_n$ has a backedge $v_jv_i$ with $j-i \ge 2$. By the first point of ($\ast$), $v_jv_i$ is a backedge of both $v_{\pi_1(1)},\ldots,v_{\pi_1(n)}$ and $v_{\pi_2(1)},\ldots,v_{\pi_2(n)}$. Thus, since these two orderings are matching orderings and one of the ends of $v_jv_i$ is $v_i$, we have that $e(v_i,v_{i+1})$ is not a backedge of either of these orderings. However, $(\ast)$ implies that any edge of the form $e(v_{i^\prime},v_{i^\prime+1})$ is a backedge of exactly one of $v_{\pi_1(1)},\ldots,v_{\pi_1(n)}$ and $v_{\pi_2(1)},\ldots,v_{\pi_2(n)}$. This is a contradiction, so $v_1,\ldots,v_n$ has no backedges $v_jv_i$ with $j-i \ge 2$.

Thus, all backedges of $v_1,\ldots,v_n$ are of the form $e(v_i,v_{i+1})$ for some $1 \le i \le n-1$. We claim that either no edge of this form is a backedge, or all edges of this form are backedges. Suppose the contrary. Then there is some $1 \le i \le n-2$ such that exactly one of $e(v_i,v_{i+1})$, $e(v_{i+1},v_{i+2})$ is a backedge. Let $j \in \{i,i+1\}$ such that $e(v_j,v_{j+1})$ is not a backedge. Then by ($\ast$), $e(v_i,v_{i+1})$ and $e(v_{i+1},v_{i+2})$ are both backedges in $v_{\pi_j(1)},\ldots,v_{\pi_j(n)}$. This is a contradiction since $v_{\pi_j(1)},\ldots,v_{\pi_j(n)}$ is a matching ordering and both these edges have end $v_{i+1}$, proving the claim.

Thus, $v_1,\ldots,v_n$ has no backedges $v_jv_i$ with $j-i \ge 2$, and either $e(v_i,v_{i+1})$ is not a backedge for all $i$, or $e(v_i,v_{i+1})$ is a backedge for all $i$. In the former case, $G$ is transitive and $v_1,\ldots,v_n$ is the ordering $1 \cdot G$. In the latter case, $G$ is $P_n$ and $v_1,\ldots,v_n$ is the ordering $1 \cdot G$. This proves the lemma.
\end{proof}

We are now ready to prove the first two points in the proposition. Suppose we have the conditions of the proposition. First, assume $\sigma = \sigma_X$. Let $u_1,\ldots,u_n$ be the ordering $v_{\pi_2(1)},\ldots,v_{\pi_2(1)}$. Then $u_{\pi_2(1)},\ldots,u_{\pi_2(n)}$ is $v_1,\ldots,v_n$, and $u_{\pi_1(1)},\ldots,u_{\pi_1(n)}$ is $v_{\pi_2(\pi_1(1))},\ldots,v_{\pi_2(\pi_1(n)}$, which is $v_{\sigma_X(1)},\ldots,v_{\sigma_X(n)}$ since $\pi_2(\pi_1(i)) = (\pi_1\pi_2)(i)$ and $\pi_1\pi_2 = \sigma_X$. Thus, both $u_{\pi_1(1)},\ldots,u_{\pi_1(n)}$ and $u_{\pi_2(1)},\ldots,u_{\pi_2(n)}$ are matching orderings of $G$. By the Lemma, $G$ is either transitive or isomorphic to $P_n$, and $u_1,\ldots,u_n$ is the ordering $1 \cdot G$. Since $v_1,\ldots,v_n$ is $u_{\pi_2(1)},\ldots,u_{\pi_2(n)}$, we have that $v_1,\ldots,v_n$ is $\pi_2 G$, as desired. The proof for $\sigma = \sigma_X^{-1}$ is analagous.

Finally, suppose $n=4$ and $\sigma = \tau_X$. Define $e(u,v)$ as before. Call and edge a \emph{long} edge if it has one end in $\{v_1,v_2\}$ and the other end in $\{v_3,v_4\}$. Then a long edge is a backedge of $v_{\sigma(1)},\ldots,v_{\sigma(4)}$ if and only if it is not a backedge of $v_1,\ldots,v_n$. Thus, each long edge is a backedge of exactly one of $v_1,\ldots,v_4$ and $v_{\sigma(1)},\ldots,v_{\sigma(4)}$. On the other hand, there are four long edges, and a matching ordering on four vertices can have at most two backedges. Thus, each of $v_1,\ldots,v_4$ and $v_{\sigma(1)},\ldots,v_{\sigma(4)}$ have exactly two backedges, both of which are long edges. So the backedges of $v_1,\ldots,v_4$ are either $\{v_3v_1, v_4v_2\}$ or $\{v_4v_1, v_3v_2\}$. In the first case $v_1,\ldots,v_n$ is the ordering $\pi_2 P_4$, and in the second case it is $\pi_2\pi_1 P_4$, as desired. 
\end{proof}

\subsection{Minimal non-matching tournaments}

A tournament $G$ is a \emph{minimal non-matching} tournament if $G$ is not a matching tournament and every subtournament of $G$ with $ < |V(G)|$ vertices is a matching tournament. Since every subtournament of a matching tournament is also a matching tournament, an equivalent definition of a minimal non-matching tournament is a tournament $G$ which is not a matching tournament and for which every subtournament of $G$ with $|V(G)|-1$ vertices is a matching tournament.

Clearly, a tournament is not a matching tournament if and only if it has a minimal non-matching subtournament. It is then natural to ask whether the list of minimal non-matching tournaments is finite. In this section we answer in the negative.

\begin{thm} \label{inftymin}
There are infinitely many minimal non-matching tournaments.
\end{thm}

\begin{proof}
For $n \ge 3$, define the tournament $Q_n$ as follows. Let the vertices of $Q_n$ be $v_1,\ldots,v_n$, and let the backedges of the ordering $v_1,\ldots,v_n$ be exactly $v_{n-1}v_1$, $v_nv_{n-2}$, and $v_{i+1}v_i$ for $1 \le i \le n-3$. Notice that the tournament induced on $v_1,\ldots,v_{n-2}$ is $P_{n-2}$.

We will prove that for all odd $n \ge 7$, $Q_n$ is a minimal non-matching tournament, which suffices to prove the theorem. We first show that $Q_n$ is not a matching tournament. Suppose $Q_n$ has a matching ordering $u_1,\ldots,u_n$. By Proposition \ref{degrees}, $u_n$ has indegree $n-2$ or $n-1$. Thus, $u_n$ must be $v_n$. Let $Q^\prime = Q_n - v_n$; then $u_1,\ldots,u_{n-1}$ is a matching ordering of $Q^\prime$. By Proposition \ref{degrees}, $u_{n-1}$ has indegree $n-3$ or $n-2$ in $Q^\prime$; thus, $u_{n-1}$ must be $v_{n-1}$. Let $Q^{\prime\prime} = Q^\prime - v_{n-1}$, so $u_1,\ldots,u_{n-2}$ is a matching ordering of $Q^{\prime\prime}$. Now, since $Q^{\prime\prime}$ is isomorphic to $P_{n-2}$ and $n \ge 7$, we have by Corollary \ref{Pnmatching} that $u_1,\ldots,u_{n-2}$ is either $\pi_1 Q^{\prime\prime}$ or $\pi_2 Q^{\prime\prime}$. However, for odd $n$, the ordering $\pi_1 Q^{\prime\prime}$ has the backedge $v_{n-2}v_{n-3}$, and the ordering $\pi_2 Q^{\prime\prime}$ has the backedge $v_2v_1$. (See the two bulleted points in the proof of Proposition \ref{Xorderings}.)
Thus, since $u_1,\ldots,u_n$ has the backedges $v_{n-1}v_1$ and $v_nv_{n-2}$, one of $v_1$, $v_{n-2}$ is an end of two backedges of $u_1,\ldots,u_n$, contradicting the fact that $u_1,\ldots,u_n$ is a matching ordering. So $Q_n$ is not a matching tournament.

Now, let $v$ be any vertex of $Q_n$. We wish to prove that $Q_n - v$ is a matching tournament. If $v = v_n$, consider the ordering $u_1,\ldots,u_{n-2},v_{n-1}$ of $Q_n - v_n$, where $u_1,\ldots,u_{n-2}$ is the $\pi_1 P_{n-2}$ ordering of $v_1,\dots,v_{n-2}$. In the $\pi_1 P_{n-2}$ ordering of $v_1,\dots,v_{n-2}$ for odd $n$, $v_1$ is not the end of any backedge. Putting $v_{n-1}$ after $u_1,\ldots,u_{n-2}$ adds only the backedge $v_{n-2}v_1$, so $u_1,\ldots,u_{n-2},v_{n-1}$ is a matching ordering, as desired. Similarly, if $v = v_{n-1}$, consider the ordering $u_1,\ldots,u_{n-1},v_n$ of $Q_n-v_{n-1}$, where $u_1,\ldots,u_{n-2}$ is the $\pi_2 P_{n-2}$ ordering of $v_1,\dots,v_{n-2}$. In the $\pi_2 P_{n-2}$ ordering of $v_1,\dots,v_{n-2}$ for odd $n$, $v_{n-2}$ is not the end of any backedge, and adding $v_n$ afterwards adds only the backedge $v_nv_{n-2}$. So $u_1,\ldots,u_{n-1},v_n$ is a matching ordering, as desired.
 
Now suppose $v = v_m$ for $1 \le m \le n-2$. Let $H_1$ be the subtournament of $Q_n$ induced on $\{v_1,\ldots,v_{m-1}\}$ and let $H_2$ be the subtournament of $Q_n$ induced on $\{v_{m+1},\ldots,v_{n-2}\}$. If one of these two sets is empty, then in the following argument we ignore all mentions to the corresponding subtournament. Now, $H_1$ and $H_2$ are isomorphic to $P_{m-1}$ and $P_{n-m-2}$, respectively. Let $u_1,\ldots,u_{m-1}$ be the ordering $\pi_1 H_1$, and let $u_{m+1},\ldots,u_{n-2}$ be the ordering $\pi_2 H_2$ if $n-m-2$ is odd and $\pi_1 H_2$ if $n-m-2$ is even. Consider the ordering 
\[
u_1,\ldots,u_{m-1},u_{m+1},\ldots,u_{n-2},v_{n-1},v_{n}
\]
of $Q_n - v_m$. The ordering $u_1,\ldots,u_{m-1},u_{m+1},\ldots,u_{n-2}$ is a matching ordering of $Q_n - \{v_m,v_{n-1},v_{n}\}$, and furthermore, we can check that each of $v_1,v_{n-2}$ is either absent from this ordering or is not the end of a backedge of this ordering. Adding $v_{n-1},v_n$ to the end of this ordering adds at most the backedges $v_{n-1}v_1$ and $v_{n-2}v_n$, so $u_1,\allowbreak\ldots,\allowbreak u_{m-1},\allowbreak u_{m+1},\allowbreak \ldots,\allowbreak u_{n-2},\allowbreak v_{n-1},\allowbreak v_{n}$ is a matching ordering of $Q_n - v_m$. This completes the proof.    
 
\end{proof}

\section{Excluding almost-transitive subtournaments}

\subsection{Almost-transitive tournaments}

The set of tournaments which do not have $I_n$ as a subtournament is finite; indeed, it is well-known that any tournament with $m$ vertices has a transitive subtournament with $\ge \log_2 m$ vertices. We can then ask what happens when we exclude subtournaments that are \emph{almost} transitive tournaments; that is, tournaments obtained from transitive tournaments by reversing a small number of edges. This final section deals with the structure of tournaments that exclude such almost-transitive tournaments.

Before proceeding, we introduce some final conventions. Given a tournament $G$ with an ordering $v_1,\ldots,v_n$ of its vertices, the \emph{length} of a backedge $v_jv_i$ is $j-i$. If $G$ is a tournament and $u_1,\ldots,u_m$ is an ordering of a subset of $V(G)$, we call $u_1,\ldots,u_m$ a \emph{subordering} of $G$; in addition, we will often use $u_1,\ldots,u_m$ to mean the subtournament of $G$ induced on $\{u_1,\ldots,u_m\}$. Finally, if $X$ is a set and $X^\circ$ is an ordering $x_1,\ldots,x_n$ of $X$, and $y_1,y_2$ are elements not in $X$, we write $y_1,X^\circ,y_2$ to denote the ordering $y_1,x_1,\ldots,x_n,y_2$.

Now, let $n \ge 3$. Let $J_n$ be the tournament with vertices $v_1,\ldots,v_n$ such that the ordering $v_1,\ldots,v_n$ has only one backedge $v_nv_1$. Let $K_n$ be the tournament with vertices $v_1,\ldots,v_n$ such that this ordering has only one backedge $v_nv_2$, and let $K_n^{\ast}$ be the tournament with vertices $v_1,\ldots,v_n$ such that this ordering has only one backedge $v_{n-1}v_1$. Note that $J_n$, $K_n$, and $K_n^\ast$ have subtournaments isomorphic to $J_m$, $K_m$, and $K_m^\ast$, respectively, for $n \ge m$.

We also have the following alternate definitions: Let $C_3$ be the cyclic triangle, let $D_4$ to be the tournament with vertices ordered as $v_1,v_2,v_3,v_4$ such that $v_1 \Ra \{v_2,v_3,v_4\}$ and $\{v_2,v_3,v_4\}$ forms a cyclic triangle, and let $D_4^\ast$ be the tournament with vertices ordered as $v_1,v_2,v_3,v_4$ such that $\{v_2,v_3,v_4\} \Ra v_1$ and $\{v_2,v_3,v_4\}$ forms a cyclic triangle.
Then $J_n$ is $C_3(I_1, I_1, I_{n-2})$, $K_n$ is $D_4(I_1,I_1,I_{n-3})$, and $K_n^\ast$ is $D_4^\ast(I_1,I_1,I_{n-3})$.

In the first subsection, we state and prove a theorem by P. Seymour on the structure of tournaments that exclude $J_n$ subtournaments. In the second subsection, we prove our main result, which is a structure theorem for tournaments that exclude both $K_n$ and $K_n^\ast$.

\subsection{Excluding $J_n$}

Seymour's theorem is as follows.

\begin{thm}[Seymour] \label{noJn}
The following hold.
\renewcommand{\labelenumi}{(\roman{enumi})}
\begin{enumerate}
\item For each $n \ge 3$, there exists $k$ such that if a tournament $G$ has no subtournament isomorphic to $J_n$, then $V(G)$ has an ordering $v_1,\ldots,v_{|V(G)|}$ that has no backedge of length $>k$. 
\item For each $k \ge 1$, there exists $n$ such that if a tournament $G$ has an ordering $v_1,\ldots,v_{|V(G)|}$ of its vertices with no backedge of length $>k$, then $G$ has no subtournament isomorphic to $J_n$.  
\end{enumerate}
\end{thm}

\begin{proof}[Proof of (i)]
We will prove that (i) holds with $k = 10n \cdot 2^{2n}$.
Let $G$ be a tournament that has no subtournament isomorphic to $J_n$. Let $I$ be a transitive subtournament of $G$ which maximizes $V(I)$, and let $u_1,\ldots,u_{V(I)}$ be the vertices of $I$ so that $u_i \ra u_j$ if $i < j$. Suppose $v \in V(G) \minus V(I)$. Let $A = A_G(v) \cap V(I)$ and $B = B_G(v) \cap V(I)$.
Since $G$ has no $J_n$ subtournament, there are no integers $1 \le i_1 < \cdots < i_{n-1} \le |V(I)|$ such that the subordering $u_{i_1},\ldots,u_{i_{n-1}},v$ has only one backedge $vu_{i_1}$, or such that the subordering $v,u_{i_1},\ldots,u_{i_{n-1}}$ has only one backedge $u_{i_{n-1}}v$. Thus, there are no integers $1 \le i_1 < \cdots < i_{n-1} \le |V(I)|$ such that either
\renewcommand{\labelenumi}{(\arabic{enumi})}
\begin{enumerate}
\item $u_{i_1} \in A$ and $u_{i_2},\ldots,u_{i_{n-1}} \in B$, or
\item $u_{i_1},\ldots,u_{i_{n-2}} \in A$ and $u_{i_n} \in B$.
\end{enumerate} 

Now, if either $A$ or $B$ is empty, then $I+v$ is transitive, contradicting the maximality of $I$. So both $A$ and $B$ are nonempty. Let $a_v = \min\{i : u_i \in A\}$ and $b_v = \max\{i : u_i \in B\}$. If $a_v > b_v$, then since $u_i \in B$ for all $i < a_v$ and $u_i \in A$ for all $i > b_v$, we must have $a_v - b_v = 1$. But then $u_1,\ldots,u_{b_v},v,u_{a_v},\ldots,u_{|V(I)|}$ is a transitive subtournament of $G$, contradicting the maximality of $I$. So $a_v < b_v$.

We now claim that $b_v - a_v < 2n$. To see this, let $Y = \{u_i : a_v < i < b_v\}$. Since $a_v \in A$, by fact (1) from above we have that $|Y \cap B| < n-1$. Similarly, since $b_v \in B$, by (2) we have that $|Y \cap A| < n-1$. Hence, $|Y| < 2n-2$, so $b_v - a_v < 2n-1 < 2n$, as claimed.

Thus, to each vertex $v \in V(G) \minus V(I)$, we have associated integers $a_v < b_v$ such that $u_i \ra v$ for all $i < a_v$, $v \ra u_i$ for all $i > b_v$, and $b_v - a_v < 2n$. In particular, $v \ra u_i$ for all $i \ge a_v + 2n$. For each $1 \le a \le |I(V)| - 1$, let $X_a$ denote the set of vertices $v \in V(G) \minus V(I)$ such that $a_v = a$. Then the sets $X_a$ form a partition of $V(G) \minus V(I)$, and for all $a$,
\begin{equation*} \tag{$\dagger$}
\{u_1,\ldots,u_{a-1}\} \Ra \{u_a\} \cup X_a \Ra \{u_{a+2n},\ldots,u_{|V(I)|}\}.
\end{equation*}
Now, we claim that $|X_a| \le 2^{2n}$ for all $a$. Indeed, suppose $|X_a| \ge 2^{2n} + 1$ for some $a$. Then the subtournament of $G$ induced on $X_a$ has a transitive subtournament with $\lceil \log_2(2^{2n} + 1) \rceil = 2n+1$ vertices. Let $X \subseteq X_a$ be the vertices of this transitive subtournament. By ($\dagger$), the subtournament of $G$ induced on 
\[
\{u_1,\ldots,u_{a-1}\} \cup X \cup \{u_{a+2n},\ldots,u_{|V(I)|}\}
\]
is transitive. However, since $|X| = 2n+1$, this transitive tournament has $|V(I)|+1$ vertices, contradicting the maximality of $I$. So $|X_a| \le 2^{2n}$ for all $a$, as claimed.

Now, for each $a$, let $X_a^\circ$ be an arbitrary ordering of $X_a$. (The ordering is empty if $X_a$ is empty.) We claim that the ordering
\begin{equation*} \tag{$\ast$}
u_1, X_1^\circ, u_2, X_2^\circ, u_3, X_3^\circ, \ldots, u_{|V(I)| - 1}, X_{|V(I)|-1}^\circ, u_{|V(I)|}
\end{equation*}
of $V(G)$ has no backedge of length $> 10n \cdot 2^{2n}$. This will suffice to prove (i).

Suppose $v^\prime v$ is a backedge of ($\ast$). Let $a$ be such that $v \in \{u_a\} \cup X_a$, and let $j$ be such that $v^\prime \in \{u_j\} \cup X_j$. First suppose that $j \le a+3n-1$. Then the set of vertices between $v$ and $v^\prime$ in ($\ast$) are a subset of $X_a \cup \{u_{a+1}\} \cup X_{a+1} \cup \cdots \cup \{u_j\} \cup X_j$. Since $j \le a+3n-1$, and $|X_i| \le 2^{2n}$ for all $i$, there are thus at most $3n \cdot 2^{2n} + 3n - 1$ vertices between $v$ and $v^\prime$ in ($\ast$). Thus, the length of $v^\prime v$ is at most $3n \cdot 2^{2n} + 3n < 10n \cdot 2^{2n}$, as desired. 

Now suppose $j \ge a+3n$. By ($\dagger$), we have $v \Ra \{u_{a+2n}, u_{a+2n+1}, \ldots, u_{a+3n-3}\} \Ra v^\prime$. By assumption, we have $v^\prime \ra v$. Thus, $v,\allowbreak u_{a+2n},\allowbreak \ldots,\allowbreak u_{a+3n-3},\allowbreak v^\prime$ is isomorphic to $J_n$, a contradiction. So we cannot have such a backedge $v^\prime v$. This proves (i).
\end{proof}

\begin{proof}[Proof of (ii)]
We prove that (ii) holds with $n = 10k$. Let $G$ be a tournament that has an ordering $v_1,\ldots,v_{|V(G)|}$ of $V(G)$ with no backedges of length $>k$. Suppose that $G$ has a subtournament isomorphic to $J_{10k}$. Let $u_1,\ldots,u_{10k}$ be an ordering of the vertices of this subtournament such that $u_{10k}u_1$ is the only backedge. Now, let $i,j$ be such that $u_1 = v_i$ and $u_{10k} = v_j$. First suppose $i < j$. Take any vertex $u \in \{u_2,\ldots,u_{10k-1}\}$, and let $x$ be such that $u = v_x$. We consider the number of possible values of $x$. Since $i < j$, the edge $v_jv_i = u_{10k}u_1$ is a backedge of $v_1,\ldots,v_{|V(G)|}$, and hence has length $\le k$. So if $i < x < j$ then there are $<k$ possible values of $x$. If $x < i$, then $v_iv_x = u_1u$ is a backedge of $v_1,\ldots,v_{|V(G)|}$ and hence has length $\le k$, so there $\le k$ possible values of $x$. Similarly if $x > j$, then $v_xv_j = uu_{10k}$ is a backedge of $v_1,\ldots,v_{|V(G)|}$ and hence has length $\le k$, so there are $\le k$ possible values of $x$. In total there are $\le 3k$ possible values of $x$. This holds for all $u \in \{u_2,\ldots,u_{10k-1}\}$, a contradiction since this set has $10k-2 > 3k$ vertices.

Now suppose $j < i$. Again, take a vertex $u \in \{u_2,\ldots,u_{10k-1}\}$, let $u = v_x$, and consider the number of possible values for $x$. If $x < i$, then $v_iv_x = u_1u$ is a backedge of $v_1,\ldots,v_{|V(G)|}$ and hence has length $\le k$, so there are $\le k$ possible values for $x$. If $x > i$, then $x > j$, so $v_xv_j = uu_{10k}$ is a backedge of $v_1,\ldots,v_{|V(G)|}$, and as before there are $\le k$ possible values of $x$. In total there are $\le 2k$ possible values of $x$, which as before is a contradiction. This completes the proof.
\end{proof}

\subsection{Excluding $K_n$ and $K_n^\ast$}

The main theorem presented in this section is the following.

\begin{thm} \label{noKn}
The following hold.
\renewcommand{\labelenumi}{(\roman{enumi})}
\begin{enumerate}
\item For each $n \ge 3$, there exists $k$ such that if a tournament $G$ has no subtournament isomorphic to $K_n$ or $K_n^\ast$, then $G$ can be written as $T_r(H_1,\ldots,H_r)$, where $r \ge 1$ is odd and for each $1 \le i \le r$, the tournament $H_i$ has no subtournament isomorphic to $J_k$.
\item For each $k \ge 3$, there exists $n$ such that if a tournament $G$ can be written as $T_r(H_1,\ldots,H_r)$, where $r \ge 1$ is odd and for each $1 \le i \le r$, the tournament $H_i$ has no subtournament isomorphic to $J_k$, then $G$ has no subtournament isomorphic to $K_n$ or $K_n^\ast$.
\end{enumerate}
\end{thm}

\begin{proof}[Proof of (i)]
For each $n \ge 3$, we will prove that there exists a large enough integer $m$ such that (i) holds with $k = \max(10 \cdot 2^m, 100n \cdot 2^{2n})$. Fix $m > 0$ for now; we will increase $m$ later if needed. Let $G$ be a tournament with no subtournament isomorphic to $K_n$ or $K_n^\ast$. If $|V(G)| < 2^m$, then there is certainly an ordering of $V(G)$ with no backedge of length $> 2^m$; thus, by Theorem \ref{noJn}(ii) (more specifically, the proof of Theorem \ref{noJn}(ii)), $G$ has no subtournament isomorphic to $J_{10\cdot 2^m}$. We can thus write $G$ as $T_1(G)$ where $G$ has no subtournament isomorphic to $J_{\max(10 \cdot 2^m, 100n \cdot 2^{2n})}$, as desired.

So assume $|V(G)| \ge 2^m$. Let $I$ be a transitive subtournament of $G$ which maximizes $V(I)$, and let $u_1,\ldots,u_{V(I)}$ be the vertices of $I$ so that $u_i \ra u_j$ if $i < j$. Since $|V(G)| \ge 2^m$, we have $|V(I)| \ge m$. Now, let $X \subseteq V(G) \minus V(I)$ be the set of vertices $v \in V(G) \minus V(I)$ such that either $u_1 \ra v$ or $v \ra u_{|V(I)|}$. Let $N = V(G) \minus (V(I) \cup X)$.

Our proof consists of two main steps. First, we use an argument similar to the one in the previous proof to show that there is an ordering of $V(I) \cup X$ in which all backedges are of bounded length. Afterwards, we show that $V(I) \cup X$ and $N$ can be broken up into homogeneous sets that can be ``weaved'' together to give the desired form.

\begin{proof}[Step 1]
Let $v \in X$, and let $A = A_G(v) \cap V(I)$ and $B = B_G(v) \cap V(I)$. By the definition of $X$, either $u_1 \in B$ or $u_{|V(I)|} \in A$.
Now, since $G$ does not have a $K_n$ subtournament, there are no integers $1 \le i_1 < \ldots < i_{n-1} \le |V(I)|$ such that
\begin{itemize}
\item the suborder $u_{i_1},v,u_{i_2},\ldots,u_{i_{n-1}}$ has only one backedge $u_{i_{n-1}}v$, or
\item the suborder $u_{i_1},u_{i_2},\ldots,u_{i_{n-1}},v$ has only one backedge $vu_{i_2}$. 
\end{itemize}
Also, since $G$ has no $K_n^\ast$ subtournament, there are no integers $1 \le i_1 < \ldots < i_{n-1} \le |V(I)|$ such that
\begin{itemize}
\item the suborder $v,u_{i_1},\ldots,u_{i_{n-2}},u_{i_{n-1}}$ has only one backedge $u_{i_{n-2}}v$, or
\item the suborder $u_{i_1},\ldots,u_{i_{n-2}},v,u_{i_{n-1}}$ has only one backedge $vu_{i_1}$.
\end{itemize}
In total, there are no integers $1 \le i_1 < \ldots < i_{n-1} \le |V(I)|$ such that any of the following hold.
\renewcommand{\labelenumi}{(\arabic{enumi})}
\begin{enumerate}
\item $u_{i_1} \in B$, $u_{i_2},\ldots,u_{i_{n-2}} \in A$, and $u_{i_{n-1}} \in B$.
\item $u_{i_1} \in B$, $u_{i_2} \in A$, $u_{i_3},\ldots,u_{i_{n-1}} \in B$.
\item $u_{i_1},\ldots,u_{i_{n-3}} \in A$, $u_{i_{n-2}} \in B$, $u_{i_{n-1}} \in A$.
\item $u_{i_1} \in A$, $u_{i_2},\ldots,u_{i_{n-2}} \in B$, and $u_{i_{n-1}} \in A$.
\end{enumerate}

Now, as in the proof of Theorem \ref{noJn}(i), $A$ and $B$ are both nonempty. Let $a_v = \min\{i : u_i \in A\}$ and $b_v = \max\{i : u_i \in B\}$. Again as in the previous proof, we have $a_v < b_v$. We claim that $b_v - a_v < 2n$. Let $Y = \{u_i : a_v < i < b_v\}$. Since $v \in X$, we have either $u_1 \in B$ or $u_{|V(I)|} \in A$. Suppose $u_1 \in B$. Since $u_1,u_{b_v} \in B$, by fact (1) from above we have that $|Y \cap A| < n-2$. Also, since $u_1 \in B$ and $u_{a_v} \in A$, by fact (2) we have that $|Y \cap B| < n-2$. Thus, $|Y| < 2n-4$. Similarly, if $u_{|V(I)|} \in A$, by fact (3) we have $|Y \cap A| <  n-2$, and by fact (4) we have $|Y \cap B| < n-2$, so $|Y| < 2n-4$. In either case we have $b_v - a_v < 2n-3 < 2n$, as claimed.

Thus, as in the previous proof, to each vertex $v \in X$ we have associated integers $a_v < b_v$ such that $u_i \ra v$ for all $i < a_v$, $v \ra u_i$ for all $i > b_v$, and $b_v - a_v < 2n$. For each $1 \le a \le |I(V)| - 1$, let $X_a$ denote the set of vertices $v \in X$ such that $a_v = a$. Then the sets $X_a$ form a partition of $X$, and by the argument in the previous proof, $|X_a| \le 2^{2n}$ for all $n$.

For each $a$, let $X_a^\circ$ be an arbitrary ordering of $X_a$. We claim that the ordering
\begin{equation*} \tag{$\ast$}
u_1, X_1^\circ, u_2, X_2^\circ, u_3, X_3^\circ, \ldots, u_{|V(I)| - 1}, X_{|V(I)|-1}^\circ, u_{|V(I)|}
\end{equation*}
of $V(I) \cup X$ has no backedge of length $> 10n \cdot 2^{2n}$. Suppose $v^\prime v$ is a backedge of ($\ast$). Let $a$ be such that $v \in \{u_a\} \cup X_a$, and let $j$ be such that $v^\prime \in \{u_j\} \cup X_j$. 
If $j \le a+3n-1$, then as in the previous proof the length of $v^\prime v$ is at most $10n \cdot 2^{2n}$, as desired. 
Now suppose $j \ge a+3n$. As in the previous proof, we have $v \Ra \{u_{a+2n}, u_{a+2n+1}, \ldots, u_{a+3n-4}\} \Ra v^\prime$ and $v^\prime \ra v$. Moreover, we have $u_{b_v} \Ra v \cup \{u_{a+2n}, \ldots, u_{a+3n-4}\} \cup v^\prime$.
Thus, $u_{b_v},v, u_{a+2n}, \ldots, u_{a+3n-4}, v^\prime$ is isomorphic to $K_n$, a contradiction. So we cannot have such a backedge $v^\prime v$, proving the claim and completing this step.
\end{proof}

Now, let $M = V(I) \cup X$, and let $v_1,\ldots,v_{|M|}$ be the ordering ($\ast$) of $M$. Thus, $v_1,\ldots,v_{|M|}$ has no backedges of length $>10n \cdot 2^{2n}$. So by the proof of Theorem \ref{noJn}(ii), the subtournament of $G$ induced on $M$ has no subtournament isomorphic to $J_{100n \cdot 2^{2n}}$. If $N$ is empty, then $M = V(G)$, so in this case $G$ has no subtournament isomorphic to $J_{100n \cdot 2^{2n}}$; we can then write $G$ as $T_1(G)$, where $G$ has no subtournament isomorphic to $J_{\max(10 \cdot 2^m, 100n \cdot 2^{2n})}$, as desired. If $N$ is not empty, we proceed to Step 2. 

\begin{proof}[Step 2]
Our goal is to partition $M$ into nonempty sets $M_1,M_2,\ldots,M_{p+1}$ and partition $N$ into nonempty sets $N_1,N_2,\ldots,N_p$ such that none of the subtournaments induced on $M_1,\ldots,M_{p+1},N_1,\ldots,N_p$ have a subtournament isomorphic to $J_{\max(10 \cdot 2^m, 100n \cdot 2^{2n})}$, and
$M_i \Ra M_j$ if $i < j$, $N_i \Ra N_j$ if $i<j$, $N_j \Ra M_i$ if $i \le j$, and $M_j \Ra N_i$ if $i < j$. If this is the case, then we can write $G$ as $G^\prime(M_1,N_1,M_2,N_2\ldots,M_p,N_p,M_{p+1})$, where $G^\prime$ is a tournament with vertices ordered as $t_1,\ldots,t_{2p+1}$ such that
\begin{itemize}
\item $t_i \ra t_j$ if $i < j$ and $i,j$ have the same parity.
\item $t_j \ra t_i$ if $i < j$ and $i,j$ have opposite parity.
\end{itemize}
Considering the vertices of $G^\prime$ in the order $t_{2p+1},t_{2p},\cdots,t_1$, we see that $G^\prime$ is in fact a BB weave. By Proposition \ref{weave}, $G^\prime$ is thus isomorphic to $T_{2p+1}$. So $G$ is of the desired form, which will complete the proof.

Since $v_1,\ldots,v_{|M|}$ is the ordering ($\ast$) of $M$, we have $v_1 = u_1$ and $v_{|M|} = u_{|V(I)|}$. Thus, by the definition of $N$, we have $N \Ra v_1$ and $v_{|M|} \Ra N$. It follows that the subtournament of $G$ induced on $N$ does not have a subtournament isomorphic to $J_{n-1}$; if there was such a subtournament $H$, then $H+v_1$ would be isomorphic to $K_n$, a contradiction.

Now, for each $w \in N$, let $A_w = A_G(w) \cap M$ and $B_w = B_G(w) \cap M$. Then $v_1 \in A_w$ and $v_{|M|} \in B_w$ for all $w \in N$.
We prove the following claim.

\setcounter{claim}{0}
\begin{claim} \label{strongseparation}
For all $w \in N$, $A_w \Ra B_w$.
\end{claim} 

\begin{proof}
Let $w \in N$. Since $v_1 \in A_w$ and $v_{|M|} \in B_w$, $A_w$ and $B_w$ are nonempty. Let $b = \min\{i : v_i \in B_w\}$ and $a = \max\{i : v_i \in A_w\}$. We claim that either $b > 200n^2 \cdot 2^{2n}$ or $a < |M| - 200n^2 \cdot 2^{2n}$. Suppose the contrary.
Then $a - b \ge |M| - 400 n^2 \cdot 2^{2n}$. Since $|M| \ge |V(I)| \ge m$, we can choose $m$ large enough so that $a-b > 200n^2 \cdot 2^{2n}$.

Now, let $Y = \{v_i : b < i < a\}$, so $|Y| \ge 200n^2 \cdot 2^{2n}$.
Suppose that $|Y \cap A| \ge 100n^2 \cdot 2^{2n}$. Let $b < i_1 < i_2 < \cdots < i_{100n^2 \cdot 2^{2n}} < a$ be such that $v_{i_s} \in Y \cap A$ for all $1 \le s \le 100n^2 \cdot 2^{2n}$. Consider the subordering 
\[
v_{i_{100n \cdot 2^{2n}}}, v_{i_{200n \cdot 2^{2n}}}, v_{i_{300n \cdot 2^{2n}}}, \ldots, v_{i_{100n(n-3) \cdot 2^{2n}}}.
\]
This subordering of $v_1,\ldots,v_{|M|}$ has no backedges, because $v_1,\ldots,v_{|M|}$ has no backedge of length $>10n \cdot 2^{2n}$. For the same reason, we have 
\[
v_b \Ra \{v_{i_{100n \cdot 2^{2n}}}, v_{i_{200n \cdot 2^{2n}}}, \ldots, v_{i_{100n(n-3) \cdot 2^{2n}}}\} \Ra v_{|M|}
\] 
as well as $v_b \ra v_{|M|}$. Thus, since $v_b,v_{|M|} \in B$ and $v_{i_{100n \cdot 2^{2n}}}, \ldots, v_{i_{100n(n-3) \cdot 2^{2n}}} \in A$, we have that
\[
v_b, w, v_{i_{100n \cdot 2^{2n}}}, v_{i_{200n \cdot 2^{2n}}}, \ldots v_{i_{100n(n-3) \cdot 2^{2n}}}, v_{|M|}
\]
is isomorphic to $J_n$, a contradiction. Hence, we must have $|Y \cap A| < 100n^2 \cdot 2^{2n}$. 

An analagous argument using $J_n^\ast$ shows that $|Y \cap B| < 100n^2 \cdot 2^{2n}$. This contradicts $|Y| \ge 200n^2 \cdot 2^{2n}$. Hence, we must have either $b > 200n^2 \cdot 2^{2n}$ or $a < |M| - 200n^2~\cdot~2^{2n}$, as claimed.

Now, assume that $b > 200n^2 \cdot 2^{2n}$. Then $v_i \in A$ for all $i \le 200n^2 \cdot 2^{2n}$. Suppose that we do not have $A_w \Ra B_w$, so there are two vertices $v_i \in A$ and $v_j \in B$ with $v_j \ra v_i$. Since $v_j \in B$, we have $j > 200n^2 \cdot 2^{2n}$. Then, since $v_j \ra v_i$ and $v_1,\ldots,v_{|M|}$ has no backedge of length $>10n \cdot 2^{2n}$, we have $i > 150n^2 \cdot 2^{2n}$. Thus, the subordering
\[
v_{100n \cdot 2^{2n}}, v_{200n \cdot 2^{2n}}, v_{300n \cdot 2^{2n}}, \ldots, v_{100n(n-3) \cdot 2^{2n}}, v_j, v_i
\]
has no backedge, since $v_j \ra v_i$ and $v_1,\ldots,v_{|M|}$ has no backedge of length $>10n \cdot 2^{2n}$. But $v_{100n \cdot 2^{2n}}, \ldots, v_{100n(n-3) \cdot 2^{2n}} \in A$ (since $100n(n-3)\cdot 2^{2n} < 200n^2\cdot 2^{2n}$), and $v_j \in B$, $v_i \in A$ by definition, so we have that
\[
w, v_{100n \cdot 2^{2n}}, v_{200n \cdot 2^{2n}}, v_{300n \cdot 2^{2n}}, \ldots, v_{100n(n-3) \cdot 2^{2n}}, v_j, v_i
\]
is isomorphic to $J_n^\ast$. This is a contradiction. Hence, if $b > 200n^2 \cdot 2^{2n}$, then $A_w \Ra B_w$. By an analagous argument using $J_n$, if $a < |M| - 200n^2 \cdot 2^{2n}$, then $A_w \Ra B_w$. Either way, $A_w \Ra B_w$, as desired.
\end{proof}

Now, the subtournament of $G$ induced on $M$ can be written as $I_s(S_1,\ldots,S_s)$, where $S_1,\ldots,S_m$ are the strong components of $G$. Since $A_w \Ra B_w$ for each $w \in N$, we have for each $w \in N$ that $A_w = V(S_1) \cup V(S_2) \cup \cdots \cup V(S_{s_w})$ and $B_w = V(S_{s_w+1}) \cup \cdots \cup V(S_s)$ for some $s_w$. It follows that for every $w,w^\prime \in N$, either $A_w \subseteq A_{w^\prime}$ or $A_{w^\prime} \subseteq A_w$; in other words, the $\subseteq$ relation on $\{A_w\}_{w\in N}$ gives a total order on $N$.
By defining an equivalence relation $\sim$ on $N$ so that $w \sim w^\prime$ if and only if $A_w = A_{w^\prime}$, 
we can thus partition $N$ into equivalence classes so that the total order on $N$ becomes a strict total order on the equivalence classes; in other words, we can partition $N$ into
nonempty sets $N_1,N_2,\ldots,N_p$ such that
$A_w = A_{w^\prime}$ for every $w,w^\prime \in N_i$, and $A_w \subsetneq A_{w^\prime}$ if $w \in N_i$, $w^{\prime} \in N_j$ for $i < j$. 

For each $1 \le i \le p$, define $A_i$ to equal $A_w$ for any $w \in N_i$. Thus, $A_1 \subsetneq A_2 \subsetneq \cdots \subsetneq A_p$. 
Now, let $M_1 = A_1$, let $M_i = A_i \minus A_{i-1}$ for $2 \le i \le p$, and let $M_{p+1} = M \minus A_p$. Then $M_i$ is nonempty for $2 \le i \le p$, and $M_1$, $M_{p+1}$ are nonempty because $v_1 \in M_1$ and $v_{|M|} \in M_{p+1}$ (since $v_1 \in A_w$ and $v_{|M|} \notin A_w$ for all $w \in N$). Thus, $M_1,\ldots,M_{p+1}$ is a partition of $M$.

In summary, we have a partition $M_1,\ldots,M_{p+1}$ of $M$ and a partition $N_1,\ldots,N_p$ of $N$ so that for all $i$, $N_i \Ra M_1 \cup M_2 \cup \cdots \cup M_i$ and $M_{i+1} \cup \cdots \cup M_{p+1} \Ra N_i$. 
By Claim 1, $A_i \Ra M \minus A_i$ for all $i$, from which it follows that $M_i \Ra M_j$ if $i < j$. We also have the following.

\begin{claim}
If $i < j$, then $N_i \Ra N_j$.
\end{claim}

\begin{proof}
Let $1 \le i < j \le p$, and suppose there are vertices $w_i \in N_i$, $w_j \in N_j$ such that $w_j \ra w_i$. Then we have the following relations.
\renewcommand{\labelenumi}{(\alph{enumi})}
\begin{enumerate}
\item $w_j \Ra \{w_i\} \cup M_1 \cup \cdots \cup M_j$
\item $w_i \Ra M_1 \cup \cdots \cup M_i \Ra M_{i+1} \cup \cdots \cup M_j \Ra w_i$
\item $\{w_j\} \cup M_{i+1} \cup \cdots \cup M_{p+1} \Ra w_i$
\item $w_j \Ra M_{i+1} \cup \cdots \cup M_j \Ra M_{j+1} \cup \cdots \cup M_{p+1} \Ra w_j$
\end{enumerate}

Relations (a) and (b) imply that $G$ contains the subtournament 
\[
D_4(w_j, w_i, M_1 \cup \cdots \cup M_i, M_{i+1} \cup \cdots \cup M_j)
\]
where we are abusing notation and using sets of vertices to mean the subtournaments of $G$ induced on them. (Note that the above expression is valid because all four sets used as arguments to $D_4(\cdot,\cdot,\cdot,\cdot)$ are nonempty.) It follows that if either $M_1 \cup \cdots \cup M_i$ or $M_{i+1} \cup \cdots \cup M_j$ has a transitive subtournament with $n-3$ vertices, then $G$ has a subtournament isomorphic to $D_4(I_1,I_1,I_1,I_{n-3}) = K_n$, a contradiction. Similarly, (c) and (d) imply that $G$ contains the subtournament
\[
D_4^\ast(w_i, w_j, M_{i+1} \cup \cdots \cup M_j, M_{j+1} \cup \cdots \cup M_{p+1}).
\] 
($M_{j+1} \cup \cdots \cup M_{p+1}$ is nonempty because $j \le p$.) If either $M_{i+1} \cup \cdots \cup M_j$ or $M_{j+1} \cup \cdots \cup M_{p+1}$ has a transitive subtournament with $n-3$ vertices, then $G$ has a subtournament isomorphic to $D_4^\ast(I_1,I_1,I_1,I_{n-3}) = K_n^\ast$, a contradiction.

Thus, none of the subtournaments induced on $M_1 \cup \cdots \cup M_i$, $M_{i+1} \cup \cdots \cup M_j$, and $M_{j+1} \cup \cdots \cup M_{p+1}$ has a transitive subtournament with $n-3$ vertices. But these three sets form a partition of $M$, and since $|M| \ge m$, for large enough $m$ at least one of these sets has size $\ge 2^n$. The subtournament induced on this set has a transitive subtournament with $n-3$ vertices, a contradiction. So no such $w_i,w_j$ exist, and hence $N_i \Ra N_j$, as desired.
\end{proof}

Now, we have partitioned $V(G)$ into nonempty sets $M_1,\ldots,M_{p+1},N_1,\ldots,N_p$ such that $M_i \Ra M_j$ if $i < j$, $N_i \Ra N_j$ if $i<j$, $N_j \Ra M_i$ if $i \le j$, and $M_j \Ra N_i$ if $i < j$. Furthermore, as we noted earlier, the subtournament of $G$ induced on $M$ has no subtournament isomorphic to $J_{100n \cdot 2^{2n}}$, and the subtournament of $G$ induced on $N$ has no subtournament isomorphic to $J_{n-1}$. Hence, none of the subtournaments of $G$ induced on $M_1,\ldots,M_{p+1},N_1,\ldots,N_p$ have a subtournament isomorphic to $J_{\max(10 \cdot 2^m, 100n \cdot 2^{2n})}$. We have thus achieved our goal, and the proof of (i) is complete.

\end{proof}

\renewcommand{\qedsymbol}{}
\end{proof}

\begin{proof}[Proof of (ii)]
We prove that (ii) holds with $n = k+1$. Let $G$ be a tournament that can be written as $T_r(H_1,\ldots,H_r)$, where $r \ge 1$ is odd and for each $1 \le i \le r$, $H_i$ has no subtournament isomorphic to $J_k$. 
If $r = 1$, then $G = H_1$, so $G$ has no subtournament isomorphic to $J_k$; thus, $G$ certainly has no subtournament isomorphic to $K_{k+1}$ or $K_{k+1}^\ast$, since $J_k$ is a subtournament of these tournaments. 
So assume $r \ge 3$. Let $r = 2\ell+1$. If $G$ has a subtournament isomorphic to $K_{k+1}$, there must be some vertex $v$ of $G$ whose outneighborhood contains a subtournament isomorphic to $J_k$. However, the outneighborhood of every $v \in V(G)$ is of the form $I_\ell(H_i, H_{i+1}, \ldots, H_{i+\ell-1})$, where the indices of the $H$ subtournaments are taken modulo $r$. Since $J_k$ is strongly connected and each $H_j$ does not have a $J_k$ subtournament, $I_\ell(H_i, H_{i+1}, \ldots, H_{i+\ell-1})$ does not have one either. Thus $G$ does not contain $K_{k+1}$. By an analagous argument, $G$ does not have a subtournament isomorphic to $K_{k+1}^\ast$, completing the proof.
\end{proof}

\section{Further questions}

To conclude, we list a few directions that future research might take. First, it would be nice to find a good application for Theorem \ref{grow}. As noted in the Introduction, Schmerl and Trotter's Theorem \ref{critical} has often been used in the study prime tournaments. Since Theorem \ref{grow} appears to be a substantive improvement on this theorem, it may help in proving further facts about the prime tournaments. Our proof of Theorem \ref{diamond} gives an example of how the theorem might be used.

In Section 5, we looked at tournaments whose vertices can be ordered so that the backedges form a matching. One can generalize this and look at tournaments with orderings in which the backedges form different structures; for example, forests. It would be good to know bounds on the number of such orderings that a prime tournament can have, and whether there are infinitely many minimal tournaments which do not have such orderings.

Finally, there is much work to be done on looking at families of tournaments which exclude certain subtournaments. Chudnovsky and Seymour very recently proved a structure theorem for the tournaments that exclude both of the following two subtournaments: for odd $n \ge 3$, $n = 2k+1$, let $L_n$ be the tournament with vertices $v_1,\ldots,v_n$ so that this ordering has only one backedge $v_n v_{k+1}$, and let $L_n^\ast$ be the tournament with vertices $v_1,\ldots,v_n$ so that this ordering has only one backedge $v_{k+1}v_1$. Their result centers around tournaments that are obtained from $T_r(I^1,\ldots,I^r)$, where $r \ge 1$ is odd and $I^1,\ldots,I^r$ are transitive tournaments, by reversing edges which are in a sense ``extreme.'' Another, somewhat different result by Latka \cite{L} gives the structure of tournaments which exclude $W_5$; in particular, the prime tournaments which do not have $W_5$ subtournaments are precisely $I_2$, $T_n$ and $U_n$ for odd $n \ge 1$, the Paley tournament with 7 vertices, and the tournament obtained from the Paley tournament with 7 vertices by deleting a vertex.

It would be nice to have a better understanding of the tournaments which exclude $T_5$ and $U_5$ subtournaments. We also want to consider tournaments which exclude the following tournament: for $n = 3k+2$, $k \ge 1$, let $E_n$ be the tournament with vertices $v_1,\ldots,v_n$ so that this ordering has only one backedge $v_{2k+2}v_{k+1}$. Part of the difficulty of this latter problem comes from our lack of understanding of even the case $E_5$. A structural theorem for the tournaments that exclude $E_5$ would not only aid in this problem, but also be interesting in itself.

\newpage
\pagestyle{empty}

\vspace*{\fill}
This thesis represents my own work in accordance with University regulations.
\vspace{2in}
\vspace*{\fill}

\end{document}